\documentclass{amsart}

\usepackage{amsmath,amssymb,enumerate}
\usepackage{mathptmx}

\usepackage[all]{xy}
\usepackage{amsthm}
\usepackage{amssymb} 
\usepackage{amsmath,amscd} 
\usepackage{mathrsfs}
\usepackage{color}
\usepackage{enumerate}

\def\GrMod{\operatorname{\mathsf{GrMod}}}

\def\turn!{\textup{!`}}

\def\op{\textup{op}}

\def\pd{\mathop{\mathrm{pd}}\nolimits}
\def\grpd{\mathop{\mathrm{gr.pd}}\nolimits}
\def\injdim{\mathop{\mathrm{id}}\nolimits}
\def\grinjdim{\mathop{\mathrm{gr.id}}\nolimits}

\def\ac{\operatorname{ac}}

\def\tot{\operatorname{\mathsf{tot}}}

\def\mrb{\mathrm{b}}

\def\Spec{\operatorname{Spec}}

\def\thick{\mathop{\mathsf{thick}}\nolimits}

\def\kk{{\mathbf k}}

\def\NN{{\mathbb N}}

\def\QQ{{\mathbb Q}}

\def\ZZ{{\mathbb Z}}

\def\cD{{\mathcal D}}
\def\cE{{\mathcal E}}

\def\cI{{\mathcal I}}

\def\cP{{\mathcal P}}

\def\sfC{{\mathsf{C}}}
\def\sfD{{\mathsf{D}}}

\def\sfH{{\mathsf{H}}}

\def\sfK{{\mathsf{K}}}

\def\tuH{{\textup{H}}}

\def\frkp{{\mathfrak{p}}}
\def\frkq{{\mathfrak{q}}}

\def\id{\operatorname{id}}
\def\op{\operatorname{op}}

\def\mod{\operatorname{mod}}
\def\Mod{\operatorname{Mod}}
\def\GrMod{\operatorname{Mod}^{\mathbb{Z}}}

\def\Ker{\mathop{\mathrm{Ker}}\nolimits}

\def\GrProj{\operatorname{Proj}^{\mathbb{Z}}}

\def\Proj{\operatorname{Proj}} 
\def\Inj{\operatorname{Inj}}

\def\Add{\operatorname{Add}}

\def\Hom{\operatorname{Hom}}

\def\Ext{\operatorname{Ext}}

\newcommand{\RHom}{\operatorname{\Bbb{R}Hom}}

\newcommand{\lotimes}{\otimes^{\mathbb{L}}}

\newcommand{\cone}{\operatorname{\mathsf{cn}}}

\def\RHom{\operatorname{\mathbb{R}Hom}}

\newtheorem{lemma}{Lemma}[section]
\newtheorem{proposition}[lemma]{Proposition}
\newtheorem{theorem}[lemma]{Theorem}
\newtheorem{corollary}[lemma]{Corollary}

\theoremstyle{definition}

\newtheorem{remark}[lemma]{Remark}
\newtheorem{example}[lemma]{Example}
\newtheorem{definition}[lemma]{Definition}
\newtheorem{conjecture}[lemma]{Conjecture}

\theoremstyle{remark}

\def\frkm{{\mathfrak{m}}}
\def\champ{{\mathop{\mathrm{amp}}}}
\def\depth{{\mathop{\mathrm{depth}}}}
\def\chdepth{{\mathop{\mathrm{chdepth}}}}
\def\inf{{\mathop{\mathrm{inf}}}}
\def\sup{{\mathop{\mathrm{sup}}}}

\def\Supp{\mathop{\mathrm{Supp}}\limits}
\def\supp{\mathop{\mathrm{supp}}\limits}

\def\wslash {{/\hspace{-4pt}/}}

\def\ulfrkp{{\underline{\frkp}}}
\def\ulfrkq{{\underline{\frkq}}}
\def\ulfrkm{{\underline{\frkm}}}

\def\DGproj{{\mathsf{DGproj}}}
\def\DGinj{{\mathsf{DGinj}}}

\title[On CDGAs]
{
On commutative differential graded algebras 
}

\date{}

\author{Hiroyuki Minamoto
}



\begin{document}

\address{Department of Mathematics Osaka Prefecture University,\\ Sakai City, Japan}
\email{minamoto@mi.s.osakafu-u.ac.jp}

\begin{abstract}
In this paper we undertake a basic study on  connective commutative differential graded algebras (CDGA), 
more precisely, piecewise Noetherian CDGA, which is a DG-counter part of commutative Noetherian algebra. 
We establish basic results for example, Auslaner-Buchsbaum formula and Bass formula without any unnecessary assumptions. 

The key notion is the sup-projective (sppj) and inf-injective (ifij) resolutions introduced by the author,  
which are DG-versions of the projective and injective resolution for ordinary modules. 
These are different from DG-projective and DG-injective resolutions which is known DG-version of the projective and injective resolution. 
In the paper, we show that sppj and ifij resolutions are powerful tools to study DG-modules. 
Many classical result about the projective and injective resolutions can be generalized to DG-setting by using sppj and ifij resolutions. . 
Among other things we prove a DG-version of Bass's structure  theorem of a minimal injective resolution holds for a minimal ifij resolution 
and a DG-version of the Bass numbers introduced by the same formula with the classical case. 
We also prove  a structure theorem of a minimal ifij resolution of a dualizing complex $D$, 
which is completely analogues to the structure theorem of a minimal injective resolution of a dualizing complex over an ordinary commutative algebra. 

Specializing to results about a dualizing complex,  we study a Gorenstein CDGA. 
We generalize a  result by 
Felix-Halperin-Felix-Thomas and  
Avramov-Foxby  which gives  conditions that  a CDGA $R$ is Gorenstein in terms of its cohomology algebra $\tuH(R)$. 
\end{abstract}

\maketitle

\tableofcontents

\section{Introduction}
In this paper we study connective commutative DG-algebras (CDGA), 
more precisely, piecewise Noetherian CDGA, which is a DG-counter part of commutative Noetherian algebra. 
Aside from results, an important feature of the paper is that it is demonstrated that the standard  techniques for ordinary commutative algebra given in text books (e.g. \cite{Matsumura}) can be generalized to CDGAs. 
In the several previous research, to study CDGAs their own methods were separately developed. However we show that standard methods for ordinary commutative algebra can be applied to CDGAs with appropriate modifications.
Actually,  in that way we prove, for example, Auslander-Buchsbaum formula (Theorem \ref{Auslander-Buchsbaum formula}) and  Bass formula (Theorem \ref{Bass formula}) for piecewise Noetherian CDGAs without any unnecessary assumptions. 

To perform this, we need to replace the projective and injective dimensions for ordinary modules with that for DG-modules introduced by Yekutieli. 
We also need to replace the projective and injective resolutions for ordinary modules with that for DG-modules, which are called sup-projective (sppj) and inf-injective (ifij) resolutions, introduced in the previous paper \cite{coppepan}. 
Although DG-versions of the projective and injective resolutions which is called the DG-projective and DG-injective resolutions (or, cofibrant and fibrant replacements, semi-projective and semi-injective resolutions)  have been fundamental tools in research of DG-algebras and DG-modules, these are not suitable to measure the projective and injective dimensions. 
In \cite{coppepan}, the sppj and ifij resolutions are introduced and proved that, roughly speaking, their length measure the projective and injective dimensions respectively (Theorem \ref{sppj resolution theorem} and Theorem \ref{ifij resolution theorem}).

We show that many classical result of commutative algebras can be generalized to CDGAs by using Yekutieli's projective and injective dimensions and the sppj and ifij resolutions. 
(This fact supports that the sppj and  ifij resolutions are  proper generalization of the projective and  injective resolution 
other than DG-projective and DG-injective resolutions.) 
Among other things, we establish a DG-version of Bass's theory of minimal injective resolutions. 
In our DG-version, the role of indecomposable injective modules in ordinary Bass's theory  is played by  
DG-modules $E_{R}(R/\ulfrkp')$ associated to  prime ideals $\frkp \in \Spec H^{0}$ of the $0$-th cohomology algebra $\tuH^{0}(R)$.
This fact may be compatible with the view point of derived algebraic geometry that 
the base affine scheme of the derived affine scheme $\Spec R$ associated to 
a CDGA $R$ is the affine scheme $\Spec \tuH^{0}(R)$ (see e.g. \cite{Gaitsgory-Rozenblyum}). 
The Bass number $\mu_{R}^{i}(\frkp, M)$ for a DG-module $M$ is defined by the same formula for ordinary modules as below
\[
\mu^{n}(\frkp, M) := 
\dim_{\kappa(\frkp)}
 \Hom_{R_{\ulfrkp}}(
\kappa(\frkp), M_{\frkp}[n]).
\]
Then we obtain the DG-version of a structure theorem of a minimal ifij resolution. 

\begin{theorem}[{Theorem \ref{Bass injective decomposition theorem}}]
Let $M \in \sfD^{> -\infty}(R)$ and $I_{\bullet}$ a minmal ifij resolution of $M$. 
Then, $\mu^{n}(\frkp, M) = 0 $ if $n \neq i + \inf I_{-i}$ for any $i \in \NN$. 
If for $i \in \NN$ we set $n = i +\inf I_{-i}$,  then
\[
I_{-i} \cong 
\left(\bigoplus_{\frkp \in \Spec \tuH^{0}(R)} E_{R}(R/\ulfrkp')^{\oplus \mu^{ n }(\frkp, M)}\right)[-\inf I_{-i}]. 
\]
\end{theorem}

We study a dualizing complex in the sense of  Yekutieli (Definition \ref{dualizing complex definition}).  
We prove   a structure theorem of a minimal ifij resolution of a dualizing complex $D$. 
The statement is completely analogues to the structure theorem of a minimal injective resolution of 
a dualizing complex, 
which is  one of the fundamental result in  classical commutative ring theory  
 proved in \cite{Residues and Duality}, summarized, for example, in \cite[Theorem 4.2]{Foxby:Complexes}.

\begin{theorem}[{Theorem \ref{structure of minamal ifij resolution of a daulizing complex}}]
Let $R$ be a connective piecewise Noetherian CDGA.  
Assume that $R$ has a dualizing complex  $D \in \sfD(R)$ with a minimal ifij resolution $I_{\bullet}$ of length $e$.  
Then  $\tuH^{0}(R)$ is catenary and $\dim \tuH^{0}(R) < \infty$. 
If moreover we assume that $\tuH^{0}(R)$ is local, 
then  the following statements hold. 
\begin{enumerate}[(1)]

\item  $\inf I_{-i} = \inf D$ for $i = 0, \cdots, e$.  

\item $e = \injdim D =\depth D = \dim \tuH^{0}(R)$.

\item 
\[
I_{-i} = \bigoplus_{\frkp } E_{R}(R/\ulfrkp')[-\inf D] 
\]
where $\frkp$ run all prime ideals such that $i = \dim \tuH^{0}(R) - \dim \tuH^{0}(R)/\frkp$. 
 \end{enumerate}
 \end{theorem}

Finally we consider a Gorenstein CDGA.
A local piecewise Noetherian CDGA $R$ is called 
 Gorenstein if it satisfies  $\injdim R < \infty$. 
In Theorem \ref{Bass theorem}, generalizing \cite[Theorem 4.3]{FIJ} we give several equivalent condition that $R$ to be Gorenstein.

We generalize a  result by 
Felix-Halperin-Felix-Thomas \cite{FHT:Gorenstein spaces}, 
Avramov-Foxby \cite{Avramov-Foxby:Locally Gorenstein homomorphism} 
which relates Gorenstein property of $R$ and its cohomology algebra $\tuH(R)$.

\begin{theorem}[{Theorem \ref{cohomology Gorenstein theorem}}]
Let $R$ be a local piecewise Noetherian CDGA  satisfying $\tuH^{\ll 0}(R) = 0$. 
Then the following conditions are equivalent. 
\begin{enumerate}[(1)]
\item 
$R$ is Gorenstein CDGA and $H$ is CM as a graded commutative algebra.

\item 
$\tuH(R)$ is Gorenstein CDGA when it is regarded as CDGA with the trivial differential $\partial_{\tuH(R)}= 0$. 

\item 
$\tuH(R)$ is Gorenstein as an ordinary graded commutative algebra. 

\item 
The following conditions are satisfied.
\begin{enumerate}
\item $\tuH^{\inf R}(R)$ is a canonical $\tuH^{0}(R)$-module. 

\item $\tuH^{-n}(R)$ is a MCM-module over $\tuH^{0}(R)$. 

\item 
There exists an isomorphism 
$\tuH(R) \xrightarrow{\cong}  \Hom_{\tuH^{0}(R)}(\tuH(R), \tuH^{\inf R}(R))$ of graded $\tuH(R)$-modules.
\end{enumerate}

\end{enumerate}
\end{theorem}

The paper is organized as follows. 
In Section \ref{Resolutions of DG-modules}, we recall projective dimension and injective dimension for DG-modules defined by Yekutieli \cite{Yekutieli}. 
We also recall   a sppj resolution and an ifij resolution which are DG-versions of projective and injective resolution from \cite{coppepan}.

In Section \ref{CDGA}, we investigate   commutative DG-algebras.
In Section \ref{CDGA:basics}, 
we develop basic notion and 
prove their properties. 
Section \ref{CDGA:Indecomposable} investigates a structure of minimal ifij resolutions. 
First we study indecomposable object of $\cI$.  
Then we  introduce  the Bass number for $M \in \sfD^{> - \infty}(R)$ 
and show that it gives a description of a minimal ifij resolution 
as in the classical case. 
Section \ref{CDGA:dualizing} deals with dualizing complexes and 
establishes a structure theorem of a minimal ifij resolution of a dualizing complex. 
As an application, we study a Gorenstein CDGA and its cohomology algebra. 
%

In Appendix \ref{CE resolution}, 
we recall the constructions of DG-projective resolutions and DG-injective resolutions.

\subsection{Notation and convention}

The basic setup and notation are the followings.

Throughout the paper, we fix a base commutative ring $\kk$ and  (DG, graded) algebra is  (DG, graded) algebra over $\kk$.
We denote by $R =(R, \partial) $   a connective commutative DG-algebra. 
We will \emph{not} assume that $R$ is strongly commutative, i.e., $a^{2} = 0$ for a homogeneous element $a$ of odd degree.   
Recall that  ``connective" means that $\tuH^{> 0}(R) = 0$.  
We note that every connective DG-algebra $R$ is quasi-isomorphic to a DG-algebra $S$ such that $S^{>0} =0$. 
Since quasi-isomorphic DG-algebras have equivalent derived categories, 
it is harmless to assume that $R^{>0} = 0$ for our purpose. 
The symbol $R^{\#}$ denotes the underlying graded algebra of $R$. 
For a DG-$R$-module $M$, 
the symbol $M^{\#}$ denotes the underlying graded $R^{\#}$-module of $M$.

For simplicity we denote by $H := \tuH(R)$ the cohomology algebra of $R$, 
by $H^{0} := \tuH^{0}(R)$ the $0$-th cohomology algebra of $R$. 
We denote by $\GrMod H$ the category of graded $H$-modules, 
by $\Mod H^{0}$ the category of $H^{0}$-modules. 

We denotes by  $\sfC(R)$  the category of DG-$R$-modules and cochain morphisms, 
by $\sfK(R)$  the homotopy category of DG-$R$-modules 
and by  $\sfD(R)$ the derived category of DG $R$-modules. 
The symbol $\Hom$ denotes the $Hom$-space of $\sfD(R)$. 

Let $n \in \{ -\infty \} \cup \ZZ \cup \{ \infty\}$. 
The symbols $\sfD^{<n}(R)$, $\sfD^{>n}(R)$ denote the full subcategories of $\sfD(R)$ consisting of $M$ such that 
$\tuH^{\geq n}(M) = 0, \ \tuH^{\leq n}(M) = 0$ respectively. 
We set 
$\sfD^{[a,b]}(R) = \sfD^{\geq a}(R) \cap \sfD^{\leq b}(R)$ 
for $a,b \in \{ - \infty \} \cup \ZZ \cup \{ \infty\}$ such that $a \leq b$. 
We set  $\sfD^{\mrb}(R) := \sfD^{< \infty}(R) \cap \sfD^{> -\infty}(R)$. 

Since $R$ is connective, the pair $(\sfD^{\leq 0}(R), \sfD^{\geq 0}(R))$ is a $t$-structure in $\sfD(R)$, 
which is called the \textit{standard} $t$-structure. 
The truncation functors  are denote by  $\sigma^{<n}, \sigma^{>n}$.  
We identify the heart $\sfH = \sfD^{\leq 0}(R) \cap \sfD^{\geq 0}(R)$ 
 of the standard  $t$-structure with  $\Mod H^{0}$ via the functor $\Hom(H^{0}, -)$, 
 which fits into the following commutative diagram 
 \[
 \begin{xymatrix}{ 
  && \sfH \ar[drr]^{\mathsf{can}} \ar[dd]_{\cong}^{\Hom(H^{0},-)} && \\ 
 \sfD(R) \ar[urr]^{\tuH^{0}} \ar[drr]_{\Hom(R, -)}  &&&& \sfD(R) \\ 
 && \Mod H^{0},  \ar[urr]_{f_{*}} &&
 }\end{xymatrix}
 \]
 where $\mathsf{can}$ is the canonical inclusion functor 
 and $f_{*}$ is the restriction functor along a canonical projection $f: R \to H^{0}$.

For a DG-$R$-module $M\neq 0$, we set
$\inf M := \inf \{ n \in  \ZZ \mid \tuH^{n}(M) \neq 0\}$,  
$\sup M := \sup \{ n \in \ZZ \mid \tuH^{n}(M) \neq 0\}$, 
$\champ M := \sup M - \inf M$. 
In the case $\inf M > - \infty$, we use the abbreviation $\tuH^{\inf}(M) := \tuH^{\inf M }(M)$. 
Similarly in the case $\sup M < \infty $, we use the abbreviation $\tuH^{\sup}(M) := \tuH^{\sup M}(M)$. 
We formally set $\inf 0:= \infty$ and $\sup 0 := - \infty$. 

In the case where we need to indicate the DG-algebra $R$, we denote $\sup_{R} M, \inf_{R} M$ and $\champ_{R} M$.

\vspace{10pt}
\noindent
\textbf{Acknowledgment}


The author would like to thank L. Shaul for his comments on the first draft of this paper 
which helped to improve many points.  
He also thanks M. Ono for drawing  his attention to  the papers by Shaul, 
for  answering a question about  dualizing complexes 
and for comments on earlier versions of this paper.   

The author  was partially  supported by JSPS KAKENHI Grant Number 26610009.

\section{Sup-projective (sppj) resolutions and inf-injective (ifij) resolutions}\label{Resolutions of DG-modules}

\subsection{Projective dimension and sppj resolution  of DG-modules}

\subsubsection{Projective dimension}

We recall the definition of the projective dimension of $M \in \sfD(R)$ introduced by Yekutieli.

\begin{definition}[{\cite[Definition 2.4]{Yekutieli}}]
Let  $a \leq b \in \{ -\infty \} \cup \ZZ \cup \{\infty\}$. 

\begin{enumerate}[(1)] 
\item 
An object $M \in \sfD(R)$ 
is said 
to have \textit{projective concentration} $[a,b]$ 
if  the functor  $F = \RHom_{R}(M, -)$ 
sends $\sfD^{[m,n]}(R)$ to 
$\sfD^{[m -b,n -a ]}(\kk)$ 
for any $m \leq n \in \{-\infty\} \cup \ZZ \cup \{\infty\}$.
\[
F(\sfD^{[m,n]}(R)) \subset \sfD^{[m -b,n -a]}(\kk).
\]

\item 
An object $M \in \sfD(R)$ 
is said 
to have \textit{strict projective concentration} $[a,b]$
if it has projective concentration $[a,b]$ 
and does't have projective concentration $[c,d]$ 
such that $[c,d] \subsetneq [a,b]$. 

\item 
An object $M \in \sfD(R)$ 
is said 
to have projective dimension $d \in \NN$ 
if it has strict projective concentration $[a,b]$ for $a,b\in \ZZ$.
such that $d= b-a$. 

In the case where, $M$ does't have a finite interval as  projective concentration, 
it is said to have infinite projective dimension. 

We denote the projective dimension by $\pd M$.
\end{enumerate}
\end{definition}

We recall the following result from {\cite[Lemma 2.3]{coppepan}}

\begin{lemma}
If $M \in \sfD(R)$ has finite projective dimension, 
then it belongs to $\sfD^{<\infty}(R)$. 
\end{lemma}

\subsubsection{Sup-projective (sppj) resolution}

We recall the definition of a sup-projective (sppj) resolution of $M \in \sfD^{< \infty}(R)$. 
For this purpose first we need introduce the class $\cP$ of DG-$R$-modules which plays the role of ordinary projective modules for ordinary projective resolution.

\begin{definition}\label{cP definition}
We denote by  $\cP \subset \sfD(R)$  the full subcategory of 
direct summands of a direct sums of $R$. 
In other words, $\cP = \Add R$.  
\end{definition}

The basic properties of $\cP$ are summarized in the lemma below taken from \cite[Lemma 2.8]{coppepan}. 
We denote by $\Proj H^{0} \subset \Mod H^{0}$ the full subcategory of projective $H^{0}$-modules. 

\begin{lemma}\label{basic of cP lemma} 
\begin{enumerate}[(1)]
\item
For $N \in \sfD(R)$, the morphism induced from $\tuH^{0}$ is an isomorphism 
\[
\Hom(P, N) \xrightarrow{\cong} \Hom(\tuH^{0}(P), \tuH^{0}(N)).  
\]
%

\item 
The functor $\tuH^{0}$ induces an equivalence $\cP \cong \Proj H^{0}$. 
\end{enumerate}
\end{lemma}

\begin{definition}[sppj morphism and sppj resolution]\label{sppj morphism definition} 
Let $M \in \sfD^{< \infty}(R), M \neq 0$. 

\begin{enumerate}
\item 
A sppj  morphism $f: P \to M $ is a morphism in $\sfD(R)$ 
such that $P \in \cP[- \sup M]$ 
and the morphism $\tuH^{\sup M}(f)$ is surjecitve. 

\item 
A sppj morphism $f: P \to M $ is called minimal  if 
the morphism $\tuH^{\sup M}(f) $ is a projective cover.

\item 
A sppj resolution $P_{\bullet}$ of $M$ is a sequence $\{ \cE_{i} \}_{\geq 0} $ of exact triangles  
\[
\cE_{i}: M_{i+1} \xrightarrow{g_{i+1}} P_{  i } \xrightarrow{f_{i}} M_{i}
\]
with $M_{0} := M$ such that $f_{i}$ is sppj.

The following inequality folds
\[
\sup M_{i+1} = \sup P_{i+1} \leq \sup P_{i} =\sup M_{i}. 
\]

For a sppj resolution $P_{\bullet}$ with the above notations, 
we set $\delta_{i} := g_{i-1} \circ f_{i}$. 
\[
\delta_{i} :P_{ i} \to P_{i-1}.
\]
Moreover we write 
\[
 \cdots \to P_{i} \xrightarrow{\delta_{i}} P_{ i -1} \to \cdots \to P_{1} \xrightarrow{\delta_{1}} P_{0} \to M. 
\]

\item 
A sppj resolution $P_{\bullet}$ is said to have length $e$ if 
$P_{i } = 0$ for $i> e$ and 
$P_{e} \neq 0$.

\item 
A sppj resolution $P_{\bullet}$ is called minimal if $f_{i}$ is minimal 
for $i \geq 0$.

\end{enumerate}
\end{definition}

By Lemma \ref{basic of cP lemma}, 
 for any $M \in \sfD^{< \infty}(R)$, 
there exists a sppj morphism $f: P \to M$. 
More precisely, for any surjective $H^{0}$-module homomorphism $\phi: Q \to \tuH^{\sup}(M)$ with $Q \in \Proj H^{0}$, 
there exists a unique $f: P \to M$ such that $P \in \cP[-\sup M]$ satisfies  $\tuH^{\sup M}(P) \cong Q$ and 
$\tuH^{\sup M} (f) = \phi$ under this isomorphism.
Thus, in particular $M$ admits a minimal sppj morphism $f: P \to M$ if and only if $\tuH^{\sup }(M)$ admits a projective cover $Q \to \tuH^{\sup}(M)$ as a $H^{0}$-module.

An important feature of sppj resolution is, roughly speaking, that 
the ``length'' of sppj resolution of $M$ measures  the projective dimension of $M$. 
The precise statement below is extracted from \cite[Theorem 2.22]{coppepan}.

\begin{theorem}\label{sppj resolution theorem}
Let $M \in \sfD^{< \infty}(R)$ and $d \in \NN$ a natural number. 
Then  
the following conditions are equivalent 

\begin{enumerate}[(1)]
\item 
$\pd  M  = d$. 

\item 
For any sppj resolution $P_{\bullet}$, 
there exists a natural number $e \in \NN$ 
which satisfying the following properties 

\begin{enumerate}[(a)]
\item $M_{e} \in \cP[-\sup M_{e}]$. 

\item 
$d = e+ \sup P_{0} -\sup M_{e}$. 

\item 
$g_{e}$ is not a split-monomorphism. 
\end{enumerate}

\item 
$M$ has sppj resolution $P_{\bullet}$ of length $e$ 
which satisfies the following properties. 
\begin{enumerate}[(a)]
\item 
$d = e+ \sup P_{0} -\sup P_{e}$. 

\item 
$\delta_{e}$ is not a split-monomorphism. 
\end{enumerate} 
\end{enumerate}

\end{theorem}

\subsubsection{Minimal sppj resolution}

A formula for computing $\Hom(M, N[n])$ for $M \in \sfD^{< \infty}(R)$ and $N \in \Mod H^{0}$  by using a sppj resolution of $M$ 
 is given in \cite[Lemma 2.18]{coppepan}. 
We recall the following corollary from \cite[Corollary 2.28]{coppepan} for the later use.

\begin{corollary}\label{sppj resolution corollary 1}
Assume that $M \in \sfD^{< \infty}(R)$ admits a minimal sppj resolution $P_{\bullet}$. 
Then for a simple $H^{0}$-module $S$ we have 
\[
\Hom(M,S[n]) 
= 
\begin{cases}
0 & n \neq i -\sup P_{i} \textup{ for any } i \geq 0,  \\
\Hom(\tuH^{\sup}(P_{i}), S) & n = i - \sup P_{i} \textup{ for some } i \geq 0.
\end{cases}
\]
\end{corollary}

A formula for computing $\tuH^{n}(M\lotimes_{R}N)$ for $M \in \sfD^{< \infty}(R)$ and $N \in \Mod H^{0}$  by using a sppj resolution of $M$ 
 is given in \cite[Lemma 2.27]{coppepan}. 
We recall the following corollary from \cite[Corollary 2.29]{coppepan} for the later use.

\begin{corollary}\label{sppj resolution corollary 2}
Assume that $M \in \sfD^{< \infty}(R)$ admits a minimal sppj resolution $P_{\bullet}$. 
Then for a simple $(H^{0})^{\op}$-module $T$ we have 
\[
\tuH^{n}(M\lotimes_{R} T) 
= 
\begin{cases}
0 & n \neq - i +\sup P_{i} \textup{ for any } i \geq 0,  \\
\tuH^{\sup}(P_{i})\otimes_{H^{0}} T & n = -i + \sup P_{i} \textup{ for some } i \geq 0. 
\end{cases}
\]
\end{corollary}

\subsection{Injective dimension and ifij resolution of DG-modules}

\subsubsection{Injective dimension}
We recall the definition of the injective dimension of $M \in \sfD(R)$ 
introduced by Yekutieli. 

\begin{definition}[{\cite[Definition 2.4]{Yekutieli}}]
Let  $a \leq b \in \{ -\infty \} \cup \ZZ \cup \{\infty\}$. 

\begin{enumerate}[(1)] 
\item 
An object $M \in \sfD(R)$ 
is said 
to have \textit{injective concentration} $[a,b]$ 
if 
such that the functor  $F = \RHom_{R}(-, M)$ 
sends $\sfD^{[m,n]}(R)$ to 
$\sfD^{[a-n,b -m ]}(\kk)$ 
for any $m \leq n \in \{-\infty\} \cup \ZZ \cup \{\infty\}$.
\[
F(\sfD^{[m,n]}(R)) \subset \sfD^{[a-n, b-m]}(\kk).
\]

\item 
An object $M \in \sfD(R)$ 
is said 
to have \textit{strict injective concentration} $[a,b]$
if it has injective concentration $[a,b]$ 
and does't have injective concentration $[c,d]$ 
such that $[c,d] \subsetneq [a,b]$. 

\item 
An object $M \in \sfD(R)$ 
is said 
to have injective dimension $d \in \NN$ 
if it has strict injective concentration $[a,b]$ for $a,b\in \ZZ$ 
such that $d= b-a$.

In the case where, $M$ does't have a finite interval as  injective concentration, 
it is said to have infinite injective dimension. 

We denote the injective  dimension by $\id M$.
\end{enumerate}

\end{definition}

We recall the following result from \cite[Lemma 3.3]{coppepan}.

\begin{lemma}\label{201903152344}
If $M \in \sfD(R)$ has finite injective dimension, 
then it belongs to $\sfD^{>-\infty}(R)$. 
\end{lemma}

\subsubsection{The class $\cI$}\label{the class cI} 

The aim of Section \ref{the class cI} 
is to introduce the full subcategory $\cI \subset \sfD(R)$ 
which plays the role of injective module for ordinary injective resolutions. 

Since we are assuming $R^{> 0} = 0$, there exists a canonical surjection $\pi: R^{0} \to H^{0}$, 
by which we regard $H^{0}$-modules as $R^{0}$-modules. 
For an injective $H^{0}$-module $J \in \Inj H^{0}$, 
we define a DG-$R$-module $G(J)$ in the following way. 
First we take a injective-hull $E_{R^{0}}(J)$ of $J$ as $R^{0}$-module, 
then we define $G(J) := \Hom^{\bullet}_{R^{0}}(R, E_{R^{0}}(J))$ with the differential induced from that of $R$. 

We define the full subcategory $\cI \subset \sfD(R)$ 
as below 
\[
\cI := \{ \textup{ the quasi-isomorphism class of } G(J) \mid J \in \Inj H^{0}\}. 
\]

We summarize basic properties of $\cI$ which are proved in \cite[Section 3.2]{coppepan}. 

\begin{lemma}\label{basic of cI lemma}
The following assertions hold. 

\begin{enumerate}[(1)]
\item 
For $M \in \sfD(R), I \in \cI$, the morphism below induced from $H^{0}$ is an isomorphism  
\[
\Hom(M, I) \cong \Hom_{H^{0}}(\tuH^{0}(M), \tuH^{0}(I)) .
\]

\item 
The functor $\tuH^{0}: \cI \to \Inj H^{0}$ gives an equivalence 
whose quasi-inverse functor is given by $G$ followed by taking the quasi-isomorphism class. 
In particular we have $\tuH^{0}(G(J)) = J$. 
\end{enumerate}
\end{lemma}

\subsubsection{Inf-injective (ifij) resolutions}
We introduce the notion of an inf-injective (ifij)-resolution of $M \in \sfD^{> -\infty}(R)$ 
and show its basic properties. 
Since almost all the proofs are analogous to that for the similar statement of sppj resolution, 
we omit them.

\begin{definition}[ifij morphism and ifij resolution]\label{ifij morphism definition} 
Let $M \in \sfD^{>-\infty}(R), M \neq 0$. 

\begin{enumerate}
\item 
A ifij morphism $f: M \to I $ is a morphism in $\sfD(R)$ 
such that $I \in \cI[- \inf M]$ 
and the morphism $\tuH^{\inf M}(f)$ is injective. 

\item 
A ifij morphism $f: M \to I$ is called minimal 
if 
the morphism $\tuH^{\inf M}(f) $ is an injective envelope.

\item 
A ifij resolution $I_{\bullet}$ of $M$ is a sequence $\{\cE_{ -i}\}_{i \geq 0}$ of exact triangles  
\[
\cE_{-i}: M_{-i} \xrightarrow{f_{-i}} I_{-i} \xrightarrow{g_{i}} M_{-i-1}
\]
with $M_{0} := M$  such that $f_{-i}$ is ifij.

The following inequality folds
\begin{equation}\label{201903141554}
\inf M_{-i} = \inf I_{-i} \leq \inf I_{-i-1} = \inf M_{-i-1}. 
\end{equation}

For an ifij resolution $I_{\bullet}$ with the above notations, 
we set $\delta_{-i} := f_{-i-1} \circ g_{-i}$. 
\[
\delta_{-i} :I_{-i} \to I_{-i-1}.
\]
Moreover we write 
\[M \to I_{0} \xrightarrow{\delta_{0}} 
I_{-1} \xrightarrow{\delta_{-1}}  \cdots \to I_{-i}  \xrightarrow{\delta_{-i}} I_{-i-1} \to \cdots. 
\]

\item 
A ifij resolution $I_{\bullet}$ is said to have length $e$ if 
$I_{-i} = 0$ for $i> e$ and 
$I_{-e} \neq 0$. 

\item 
A ifij resolution $I_{\bullet}$ is called minimal if $f_{i}$ is minimal 
for $i \leq 0$.
\end{enumerate}
\end{definition}

By Lemma \ref{basic of cI lemma}, 
 for any $M \in \sfD^{< \infty}(R)$, 
there exists an ifij morphism $f: M \to I$. 
More precisely, for any injective $H^{0}$-module homomorphism $\phi: \tuH^{\inf}(M) \to J$ with $J \in \Inj H^{0}$, 
there exists a unique $f: M \to I$ with $I = G(J)[-\inf M]$ which satisfies  $\tuH^{\sup M} (f) = \phi$.
Thus, in particular $M$ admits a minimal ifij morphism $f: M \to I$ if and only if $\tuH^{\inf }(M)$ admits an injective envelope $\tuH^{\inf}(M)$ as a $H^{0}$-module.

The following theorem extracted from \cite[Theorem 3.21]{coppepan} 
is a ifij version of Theorem \ref{sppj resolution theorem}.

\begin{theorem}\label{ifij resolution theorem}
Let $M \in \sfD^{>-\infty}(R)$ and $d$ a natural number. 
Set $F :=\RHom(-,M) $. 
Then  
the following conditions are equivalent 

\begin{enumerate}[(1)]
\item 
$\injdim  M  = d$.

\item 
For any ifij resolution $I_{\bullet}$, 
there exists a natural number $e \in \NN$ 
which satisfying the following properties.  

\begin{enumerate}[(a)]
\item 
$M_{e} \in \cI[-\inf M_{e}]$. 

\item 
$d = e + \inf M_{-e} - \inf I_{0}$. 

\item 
$g_{-e}$ is not a split-epimorphism. 
\end{enumerate}

\item 
$M$ has ifij resolution $I_{\bullet}$ of length $e$ 
which satisfies the following properties.

\begin{enumerate}[(a)]
\item $d = e + \inf I_{-e} - \inf I_{0}$.

\item 
$\delta_{e}$ is not a split-epimorphism. 
\end{enumerate} 
\end{enumerate}
\end{theorem}

\subsubsection{Minimal ifij resolution}

A formula for computing $\Hom(N, M[n])$ for $M \in \sfD^{> - \infty}(R)$ and $N \in \Mod H^{0}$  by using a ifij resolution of $M$ 
 is given in \cite[Lemma 3.20]{coppepan}. 
As a  corollary we obtain the following result which is a ifij version of Corollary \ref{sppj resolution corollary 1}

\begin{corollary}\label{ifij resolution corollary 1}
Assume that $M \in \sfD^{> -\infty}(R)$ admits a minimal ifij resolution $I_{\bullet}$. 
Then for a simple $H^{0}$-module $S$ we have 
\[
\Hom(S, M[n]) 
= 
\begin{cases}
0 & n \neq i +\inf I_{-i} \textup{ for any } i \geq 0,  \\
\Hom(S, \tuH^{\inf}(I_{i})) & n = i + \inf I_{-i} \textup{ for some } i \geq 0.
\end{cases}
\]
\end{corollary}

\section{Commutative DG-algebras}\label{CDGA}

In Section \ref{CDGA}, we deal with a  connective commutative  DG-algebra (CDGA) $R$.  
Moreover, a CDGA $R$ is assumed to be piecewise Noetherian and that $R^{>0}= 0$.  
Consequently, every homogeneous element $x \in R^{0}$ of degree $0$ is cocycle.

\subsection{Basics}\label{CDGA:basics}

In Section \ref{CDGA:basics}, 
we develop basic notion of 
piecewise Noetherian CDGA $R$ and 
its derived categories with finitely generated  cohomology groups.

Recall that 
 a CDGA $R$ is called  
\textit{ piecewise Noetherian} 
if $H^{0}$ is right Noetherian and $H^{-i}$ is finitely generated as a right $H^{0}$-module 
for $i \geq 0$. 
The name is taken from \cite{Avramov-Halperin}. 
The same notion is called  cohomological pseudo-Noetherian in \cite{Yekutieli} 
and  Noetherian in \cite{Shaul:Homological, Shaul:Injective}.

Let $\mod H^{0} \subset \Mod H^{0}$ denotes the full subcategory of finite $H^{0}$-modules. 
We denote by $\sfD_{\mod H^{0}}(R) \subset \sfD(R)$ be the full subcategory 
consisting $M$ such that $\tuH^{i}(M) \in \mod H^{0}$ for $i \in \ZZ$.  
We denote by $\sfD(R)_{\textup{fpd}}, \sfD(R)_{\textup{fid}}$ the full subcategories of DG-$R$-modules of finite projective dimensions and 
finite injective dimensions respectively. 
We set 
\[
\sfD_{\mod H^{0}}^{\square}(R) :=
\sfD^{\square}(R) \cap
\sfD_{\mod H^{0}}(R)
\] 
where $\square = [a,b] , $ etc. 
We define $\sfD^{\square}_{\mod H^{0}}(R)_{\textup{fpd}}, \sfD^{\square}_{\mod H^{0}}(R)_{\textup{fid}}$ similarly.

We will tacitly use the following finiteness of $\Hom$-spaces.

\begin{proposition} 
If $M \in \sfD^{> -\infty}_{\mod H^{0}}(R)$ and $N \in \sfD^{< \infty}_{\mod H^{0}} (R)$, 
then $\Hom(N, M[n])$  is finitely generated as $H^{0}$-module for $n \in \ZZ$. 
\end{proposition} 

\begin{proof}
We take a sppj resolution $\{\cE_{i}\}_{i \geq 0}$ of $N$ 
such that each $P_{i}$ belongs to $\cP_{\mod H^{0}}$.
\[
\cE_{i}: N_{ i+1} \to P_{i} \to N_{i}.  
\] 
Then, we have 
an exact sequence 
\[
 \Hom(N_{i+1}, M[n-1]) \to \Hom(N_{i}, M[n]) \to \Hom(P_{i},M[n]) 
\]
for $i \geq 0, n \in \ZZ$. 
Since $\Hom(P_{i}, M)$ is finitely generated $H^{0}$-module, 
the middle term $\Hom(N_{i}, M[n])$ is finitely generated 
if and only if so is the left term $\Hom(N_{i+1}, M[n-1])$.
Thus it is enough to show that $\Hom(N_{i+j}, M[n-j])$ is finitely generated for some  $j \geq 0$. 

Since $M \in \sfD^{> -\infty}(R)$ and $\sup N_{i} \leq \sup N < \infty$ for $i \geq 0$, 
we have $\Hom(N_{i+j}, M[n-j]) = 0$ and in particular it is finitely generated for $j \gg 0$.   
This completes the proof. 
\end{proof}

\subsubsection{Bound of projective dimension}

A piecewise Noetherian CDGA $R$ is called \emph{local} if the $0$-th cohomological algebra $H^{0}$ is local. 
In the case where  $R$ is local, we denote by $\frkm$  a maximal ideal of $H^{0}$ and by $k$ the residue field. 
Let $\pi^{0}: R^{0} \to H^{0}$ be a canonical projection. 
We set   $\ulfrkm := (\pi^{0})^{-1}(\frkm)$ and call it a \emph{maximal ideal} of a local CDGA $R$.   
We may identify $k = H^{0}/\frkm$ with $R^{0}/\ulfrkm$ and call it the \emph{residue field} of a local CDGA $R$. 

Recall that over a commutative Noetherian local algebra, 
every module has a projective cover. 
Therefore, 
if $R$ is local, then 
every object $M \in \sfD^{< \infty}(R)$ admits  a minimal sppj resolution.  
Hence, the projective dimension $\pd M$ can be measured at the residue field $k$ of $R$.  

\begin{proposition}\label{bound of projective dimension}
Let $R$ be a local piecewise Noetherian CDGA with the residue field $k$. 
For $M \in \sfD^{<\infty}_{\mod H^{0}}(R)$  and $d \in \NN$, 
the following conditions are equivalent.
\begin{enumerate}[(1)]
\item $\pd M \leq d$. 

\item $\Hom(M,k[n]) = 0$ for $n > d-\sup M$. 

\item $\tuH^{n}(M \lotimes_{R} k) = 0$ for $n< -d +\sup M$.
\end{enumerate}
\end{proposition}

\begin{proof}
The implications (1) $\Rightarrow$ (2) $\Leftrightarrow$ (3) are clear. 
We prove (2) $\Rightarrow$ (1). 
Let $P_{\bullet}$ be a minimal sppj resolution of $M$. 
Then if $i -\sup P_{i}> d- \sup M $, then $P_{i} = 0$ by the assumption and Corollary \cite[Corollary 2.28]{coppepan}. 
Hence in particular $\pd M < \infty$ by Theorem \ref{sppj resolution theorem}. 
Thus we only have to show that for $N \in \Mod H^{0}$ 
the condition $\Hom(M, N[n]) \neq 0$ implies $n \leq d -\sup M$. 
By Theorem \ref{sppj resolution theorem}, the condition 
implies that there exists $i$ such that $n = i - \sup P_{i}$ 
and $P_{i} \neq 0$. 
It follows from the first consideration  that $n \leq d -\sup M$. 
\end{proof}

\begin{corollary}\label{bound of projective dimension corollary}
Let $R, k$ be as in the above Proposition. 
For $M \in \sfD^{<\infty}_{\mod H^{0}}(R)$, 
we have the following equalities. 
\[
\pd M = \sup \RHom(M, k) + \sup M =  -\inf (M \lotimes_{R} k) + \sup M
\]
\end{corollary}

\subsubsection{Localization}

A prime ideal $\frkp \in \Spec H^{0}$ induces  prime (DG) ideals 
$\frkp', \ulfrkp, \ulfrkp'$ of $H, R^{0}$ and $R$ 
in the following way.
Let $\pi: R \to H^{0}, \pi^{0}: R^{0} \to H^{0}$ the canonical projections. 
Then we set  
\[
\begin{split}
\frkp' &:= \frkp \oplus H^{-1} \oplus H^{-2} \oplus \cdots , \\
\ulfrkp & := (\pi^{0})^{-1}(\frkp), \\
\ulfrkp' & := (\pi)^{-1}(\frkp) = \ulfrkp \oplus R^{-1} \oplus R^{-2} \oplus \cdots.  
\end{split}
\]
We may identify the following residue algebras. 
\[
H^{0}/\frkp =H/\frkp' = R^{0}/\ulfrkp = R/\ulfrkp'.
\]

By 
$R_{\ulfrkp'}$, we  denotes the localization of graded algebra $R^{\#}$ 
 with respect to  $\ulfrkp'$  
 with the differential 
\[
\partial_{R_{\ulfrkp'}} \left( \dfrac{r}{s} \right) := \dfrac {\partial_{R}(r)} {s}. 
\]
We note that the localization map $\phi: R \to R_{\ulfrkp'}$ is a DG-algebra morphism.

For a DG-$R$-module $M \in \sfC(R)$, we set $M_{\ulfrkp'} := M \otimes_{R}R_{\ulfrkp'}$. 
Then, 
\[
\begin{split}
(M_{\ulfrkp'})^{n}& = (M^{n})_{\ulfrkp}, \\
\tuH(M_{\ulfrkp'})& = \tuH(M)_{\frkp'}, 
\tuH^{n}(M_{\ulfrkp'})  = \tuH^{n}(M)_{\frkp} \textup{ for } n \in \ZZ.
\end{split}
\]
We may identify the residue fields 
\[
\kappa(\frkp) = \left( H^{0}/\frkp \right)_{\frkp} 
=\left(H/\frkp' \right)_{\frkp'} 
= \left( R^{0}/\ulfrkp \right)_{\ulfrkp} = 
\left(R/\ulfrkp' \right)_{\ulfrkp'}.
\]

\begin{proposition}\label{201712041945}
Let $M \in \sfD^{<\infty}_{\mod H^{0}}(R)$ and $N \in \sfD^{>-\infty}(R)$. 
Then the canonical morphism below is an isomorphism
\[
\Phi:\RHom_{R}(M,N)_{\ulfrkp} \xrightarrow{ \cong} \RHom_{R_{\ulfrkp}}(M_{\ulfrkp}, N_{\ulfrkp}).
\]
\end{proposition}

\begin{proof}
We set $\phi_{M,N}:= \tuH^{0}(\Phi_{M,N})$
\[
\phi_{M,N}:\Hom_{R}(M,N)_{\frkp} \to \Hom_{R_{\ulfrkp}}(M_{\ulfrkp}, N_{\ulfrkp}).
\]

We fix $N \in \sfD^{> -\infty}(R)$. 
Let $n = \inf N$. 
First we consider the case  $M = \bigoplus_{i\leq m} R^{\oplus p_{i}}[-i]$ for some $m \in \ZZ$ and $p_{i} \in \NN$.  
Then, 
\[
\begin{split}
\Hom_{R}(M,N)_{\ulfrkp} 
\cong 
\left(\bigoplus_{n \leq i \leq m}\Hom_{R}(R^{\oplus p_{i}}[-i], N)\right)_{\frkp} 
\cong 
\bigoplus_{n\leq i \leq m}\tuH^{i}(N^{\oplus p_{i}})_{\frkp} \\
\Hom_{R_{\ulfrkp}}(M_{\ulfrkp}, N_{\ulfrkp}) 
\cong 
\bigoplus_{n \leq i \leq m} \Hom_{R_{\ulfrkp}}(R_{\ulfrkp}^{\oplus p_{i}}[-i], N_{\ulfrkp}) 
\cong 
\bigoplus_{n\leq i \leq m}\tuH^{i}(N_{\ulfrkp}^{\oplus p_{i}})
\end{split}
\]
The left most terms of both equations are isomorphic 
via the morphism induced from $\phi_{M,N}$. 
Therefore, in this case $\phi_{M,N}$ is an isomorphism. 
Since $\phi_{M[i],N}$ is isomorphism, we conclude that $\Phi_{M,N}$ is an isomorphism. 

Since $\Phi_{M,N}$ is a natural transformation between exact functors. 
By divisage argument, 
we see that $\Phi_{M,N}$ is an isomorphism 
in the case $M$ is obtained from the objects of the forms 
$M = \bigoplus_{i\leq m} R^{\oplus p_{i}}[-i]$ 
by taking cones and shift and direct summands. 

Finally, 
we consider the general case. 
Let $m = \sup M $. 
We take a truncated DG-projective  resolution of $M$ truncated at $k = m -n +1$-th degree 
(see Appendix \ref{Cartan-Eilenberg projective resolution}). 
\[
  P_{k} \to P_{k -1} \to \cdots \to P_{0} \to M \to 0. 
\]
Let $T$ be the totalization of $P_{k} \to \cdots \to P_{0}$. 
Let   $L := \cone(\textsf{can})$ 
be the cone of   the canonical morphism  $\textsf{can}: T \to M$. 
Then $L,L[1]$  belong to $\sfD^{< n}(R)$ by Lemma \ref{20171124151}. 
Therefore, the induced morphisms 
$\Hom_{R}(M, N) \to \Hom_{R}(T, N)$ 
and
$\Hom_{R_{\ulfrkp}}(M_{\ulfrkp}, N_{\ulfrkp}) \to \Hom_{R_{\ulfrkp}}(T_{\ulfrkp}, N_{\ulfrkp})$ 
 are isomorphisms. 
Under these isomorphisms, the morphism 
 $\phi_{M,N}$ corresponds to $\phi_{T,N}$. 
 Since we already know that $\phi_{T,N}$ is an isomorphism, 
 we conclude that $\phi_{M,N}$ is an isomorphism.
\end{proof}

\subsubsection{Nakayama's Lemma} 

\begin{lemma}\label{Nakayama lemma}
For $M \in \sfD(R)$, we consider   the following conditions. 
\begin{enumerate}[(1)]
\item 
$M_{\frkp }= 0$. 
\item 
$M \lotimes_{R} \kappa(\frkp) = 0$. 
\end{enumerate}
Then (1) always implies (2). 
If $M \in \sfD^{< \infty}_{\mod H^{0}}(R)$, then (2) implies (1). 
\end{lemma}

\begin{proof}
The first statement is clear from the canonical isomorphism 
$ M\lotimes_{R}\kappa(\frkp)  \cong M_{\frkp}\lotimes_{R_{\frkp}}\kappa(\frkp)$. 

We prove the second statement. 
Assume $M \lotimes_{R} \kappa(\frkp) = 0$ and $M_{\frkp} \neq 0$. 
Let $s := \sup M_{\frkp}$ and set $M' := \sigma^{\leq s} M, \ M" := \sigma^{> s}M$. 
Then, since $M"_{\frkp} = 0$, we have $M"\lotimes_{R} \kappa(\frkp) = 0$. 
It follows from the exact triangle  $M' \to M \to M" \to $ that 
$M'\lotimes_{R} \kappa(\frkp)  =  0$. 
Therefore we have 
\[
\tuH^{s}(M) \otimes_{H^{0}}\kappa(\frkp) = \tuH^{s}(M') \lotimes_{H^{0}} \kappa(\frkp) = \tuH^{s}(M' \lotimes_{R} \kappa(\frkp)) = 0. 
\]
 On the other hand, 
 since $\tuH^{s}(M)_{\frkp} \neq 0$, we have $\tuH^{s}(M) \otimes_{H^{0}}\kappa(\frkp) \neq 0$, a contradiction. 
\end{proof}

\subsubsection{Support and small support}

\begin{definition}
For $M \in \sfD(R)$, 
we set 
\[
\begin{split}
\Supp M &:= \{ \frkp \in \Spec H^{0} \mid M_{p} = 0 \},  \\
\supp M  &:= \{ \frkp \in \Spec H^{0} \mid M\lotimes_{R}\kappa (\frkp )  = 0 \}. 
\end{split}
\]
\end{definition}

By Nakayama's Lemma, we immediately see that 
\begin{lemma} 
For $M \in \sfD(R)$, we have 
\[
\supp M \subset \Supp M.
\]
If $M \in \sfD^{< \infty}_{\mod H^{0}}(R)$, then the equality holds. 
\end{lemma}

\subsection{Cohomological regular sequence}

Let $x \in R^{0}$ and $M$ a DG-$R$-module. 
By $M\wslash xM$ we denote 
the cone of the multiplication map $x \times - : M \to M$. 
Namely, it is a DG-$R$-module having $M \oplus M[1]$ as the underlying graded $R$-module 
which has the differential  given as below.
\[
M\wslash xM := 
\left(M \oplus M[1], 
\begin{pmatrix}
\partial_{M} & x \\
0 & -\partial_{M}
\end{pmatrix} \right).
\]

For 
$R\wslash  xR$, 
we always take the commutative DG-algebra model $R[\xi]/(\xi^{2})$ 
with the differential 
$\partial (\xi) = x$ 
and the degree $\deg \xi = -1$.
Then the canonical morphism $R \to R\wslash xR$ which sends each element $r \in R$ to the constant polynomial 
equals to the canonical morphism $R \to \cone(x \times -)$ 
of the  cone of the morphism $x \times - : R \to R$. 
It is easy to check that  $S =R\wslash xR$ is also a piecewise Noetherian  CDGA satisfying $S^{>0} = 0$. 
We note that $M \wslash x M$ has a canonical DG-$R\wslash x R$-module structure and 
there is a canonical isomorphism
$M \wslash xM  \cong M \otimes_{R} (R\wslash x R) $.

The following lemma asserts that the quasi-isomorphism class of $R\wslash x R$ depends only on the cohomology class 
of $x \in R^{0}$. 

\begin{lemma}
Assume that two elements  $x , y \in R^{0}$ are cohomologous. 
Then, there exists a CDGA $T$ and 
quasi-isomorphisms $f: R\wslash xR \to T $ 
and $g : R \wslash  yR \to T$ 
which make the following commutative diagram 
\begin{equation}\label{201801091853}
\begin{xymatrix}{
R \ar[r] \ar[d] & R\wslash xR \ar[d]^{f}_{\wr} \\
R\wslash y R \ar[r]_{g}^{\sim} & T.
}\end{xymatrix}
\end{equation}

Moreover, if we consider the equivalences  
\[
\sfD(R\wslash xR) \cong \sfD(T) \cong  \sfD(R\wslash y R)
\]
induced from $f$ and $g$, 
then 
the DG-$R\wslash xR$-module $M \wslash x M$ 
corresponds to the DG-$R\wslash y R$-module $M\wslash yM$. 
\end{lemma}

\begin{proof}
Let $z \in R^{-1}$ be such that $\partial(z) = x -y$. 
We define a CDGA $T$ as follows. 
As a graded algebra, it is a graded commutative polynomial algebra over $R$ 
modulo the relations 
\[
\begin{split}
T &= R[\xi,\eta, \zeta^{(n)} \mid n \geq 1]/ (\textup{relations}),\\
\textup{relations}&: \xi^{2} = 0, \eta^{2} = 0, \zeta^{(n)}\zeta^{(m)} = \frac{(n+m)!}{n!m!}\zeta^{(n+m)} \textup{ for } n,m \geq 1.
\end{split}
\]
with the gradings $\deg \xi = -1, \deg \eta = -1$ and $\deg \zeta^{(n)} = -2n$ for $ n\geq 1$. 
We extend the differential of $R$ to  that of $T$ by setting  
the differential for the variables as below 
\[
\partial(\xi): = x, \partial(\eta) := y, \partial(\zeta^{(n)}) := (\xi - \eta -z)\zeta^{(n-1)} \textup{ for } n \geq 1
\]
where we set $\zeta^{(0)} := 1$.

We set $f: R\wslash xR \to T $ and $g: R\wslash y R \to T$ to be the obvious inclusions. 
It is clear that these morphisms make the commutative diagram (\ref{201801091853}). 
It only remains to prove $f$ and $g$ are quasi-isomorphisms. 

We prove that $f$ is a quasi-isomorphism. 
For this purpose, we give $T$ an exhaustive increasing filter $\{F_{n}\}_{n \geq 0}$ of 
DG-$R$-submodules of $T$ which is 
defined inductively as below 
\[
\begin{split}
F_{0} &:= R \oplus R\xi, \\F_{1}&: = F_{0} \oplus R\eta \oplus R\zeta, \\
F_{2n} &:= F_{2n -1} \oplus R\xi \eta \zeta^{(n-1)} \oplus R \xi \zeta^{(n)}, \\
F_{2n+1}& := F_{2n} \oplus R\eta \zeta^{(n)} \oplus R \zeta^{(n+1)}. 
\end{split}
\]  
We can check that 
for $n \geq 1$ the graded quotient 
$F_{n}/F_{n-1}$ is isomorphic to  the cone $\cone (\id_{R})$ of the identity map $\id_{R} : R \to R$. 
Thus, in particular $F_{n}/F_{n-1}$ is acyclic. 
Since $R\wslash xR$ is identified with $F_{0}$ via $f$, 
we conclude that $f$ is a quasi-isomorphism. 

In the same way we can prove that $g$ is a quasi-isomorphism.

The second statement is a consequence of isomorphisms below  of DG-$T$-modules 
for $M \in \sfC(R)$,  
\[
M \otimes_{R\wslash xR} T \cong M \otimes_{R}T \cong M \otimes_{R\wslash y R} T.
\]
\end{proof}

For $x \in H^{0}$, 
we define the CDGA $R\wslash xR$ to be $R\wslash \underline{x} R$ 
for some $\underline{x} \in R^{0}$ whose cohomology class is $x$. 
By the above lemma, 
we may denote $\underline{x}$ by the same symbol $x$.

The following is a DG-version of \cite[p. 140, Lemma 2]{Matsumura}.

\begin{lemma}\label{waru lemma} 
Let $M \in \sfD(R)$ and $x \in H^{0}$.   
Set $S: =R\wslash  xR$. 
Then for $N \in \sfD(S)$, we have 
\[
\Hom_{S}(N,M\wslash  xM[-1] ) \cong \Hom_{R}(N,M)
\]
where in the left hand side, we regard $N$ as an object of $\sfD(R)$ by restriction of scalar. 
\end{lemma}

\begin{proof}
Since $N \cong N \lotimes_{S} S$, we have 
\[
\RHom_{R}(N, M)  \cong \RHom_{R}(N \lotimes_{S} S, M ) \cong \RHom_{S}(N, \RHom_{R}(S, M)). 
\]
The contravariant functor $\RHom(- , M)$ sends the morphism $x \times -: R \to R$ 
to $x \times- : M \to M$. 
Thus we have $\RHom(S, M) \cong M\wslash  x M [-1]$. 
\end{proof}

We introduce the notion of a regular sequence for DG-modules 
and a cohomological depth. 
We note that for complexes of module over a commutative algebra, 
the notion of a regular sequence was already introduced by Foxby \cite[Remark 3.18]{Foxby}.

\begin{definition}
Let $M \in \sfD^{> -\infty}(R)$. 

\begin{enumerate}[(1)]
\item 
An element $ x \in R^{0}$ is called \textit{cohomological} $M$-regular 
if it is  $\tuH^{\inf}(M)$-regular, that is, 
the induced morphism $x \times -: \tuH^{\inf}(M) \to \tuH^{\inf}(M)$ is  
injective.

\item A sequence $x_{1}, x_{2}, \cdots, x_{r} \in R^{0}$ is called 
\textit{cohomological $M$-regular} 
if it satisfies the following condition. 
We define inductively $M_{i}$ for $i = 1, \dots, r$ 
as $M_{1}:= M\wslash  x_{1}M, M_{i} :=M_{i-1}\wslash   x_{i}M_{i-1}$. 
Then, $x_{i }$ is cohomological regular on $M_{i -1}$.   
\end{enumerate}
\end{definition}

\begin{definition}
Let $R$ be a piecewise Noetherian CDGA with the residue field $k$.  
Then we set 
\[
\chdepth M := \inf \RHom(k, M) - \inf M\]
and call it the \textit{cohomological depth} of $M$. 
\end{definition}

\begin{remark}
In several papers for example \cite{Foxby, Jorgensen:Amplitude}, the depth of a DG-module $M$ is defined  to be $\inf \RHom_{R}(k, M)$, which can be any integer and $\pm \infty$. 
Contrary to this our cohomological depth is always non-negative. 
Moreover, 
we will prove in Proposition \ref{cohomological depth proposition} 
that
it coincides with the maximal length of cohomological regular sequence of $M$
\end{remark}

Looking the cohomological exact sequence of 
the exact triangle 
\begin{equation}\label{20181151632}
M \xrightarrow{x} M \to M\wslash xM \to,
\end{equation} 
we can deduce the following lemma. 
 
\begin{lemma}\label{amplitude lemma} 
Let $R$ be a piecewise Noetherian CDGA with a maximal ideal $\ulfrkm$.  
Let $M$ be a non-zero object of $\sfD^{> -\infty}_{\mod H^{0}}(R)$  and 
$x\in \ulfrkm$. 
Then, the following statements hold
\begin{enumerate}[(1)]
\item
$\inf M -1 \leq \inf (M\wslash xM) \leq \inf M$. 

\item
$\inf (M\wslash xM) = \inf M$ 
if and only if 
the element $x$ is cohomological $M$-regular. 
\end{enumerate}
\end{lemma}

\begin{proof}
The first inequality of (1) can be deduced from 
the long cohomology exact sequence of (\ref{20181151632}). 
Set $b := \inf M$. 
By   Nakayama's Lemma, the map $x: \tuH^{b}(M) \to \tuH^{b}(M)$ is not surjective. 
Therefore  $\tuH^{b+1}(M) \neq 0$. 
This proves the second inequality of (2). 

We have $\inf M\wslash xM = \inf M$ if and only if the map 
$x: \tuH^{b}(M) \to \tuH^{b}(M)$ is injective. 
The latter condition is satisfied precisely when 
$x$ is cohomological $M$-regular. 
\end{proof}

The relationship between the depth and the maximal regular sequence 
has cohomological version.

Combining Lemma \ref{waru lemma} and Lemma \ref{amplitude lemma}, 
we obtain 
\begin{lemma}\label{cdp lemma}
Let $R$ be a piecewise Noetherian CDGA with a maximal ideal $\ulfrkm$  and the residue field $k$. 
Let $M \in \sfD^{> -\infty}_{\mod H^{0}}(R)$ and $x$  a cohomological $M$-regular element in $\frkm'$.  
Then, 
$\chdepth_{R\wslash xR} (M\wslash  xM) = \chdepth M  -1$.
\end{lemma}

The next proposition justify the terminology ``the cohomological depth''. 
\begin{proposition}\label{cohomological depth proposition}
Let $R$ be a piecewise Noetherian CDGA  with the residue field $k$. 
Let $M \in \sfD^{> -\infty}_{\mod H^{0}}(R)$. 
The cohomological depth $\chdepth M $
 coincides with the maximal length of 
cohomological $M$-regular sequence in $\ulfrkm$. 
\end{proposition}

\begin{proof}
We prove the statement by induction on $c := \chdepth( M)$. 
Let $r$ be the maximal length in question.  

In the case $c= 0$. 
Then since 
\[
\Hom_{H^{0}}(k , \tuH^{\inf}(M) ) 
= \Hom(k, M [\inf M] )  \neq 0,
\]
there exists no $\tuH^{\inf}(M)$-regular element in $\frkm$ by \cite[Theorem 16.6]{Matsumura}. 
This shows that $r= 0$.

Assume that $ c> 0$ and the case $c -1$ is already proved. 
Since
\[\Hom(k , \tuH^{\inf}(M) ) 
= \Hom(k, M [\inf M] )  =  0,\]
there exists a $\tuH^{\inf}(M)$-regular element $x$ in $\frkm'$ by \cite[Theorem 16.6]{Matsumura}. 
Then by Lemma \ref{cdp lemma}, $\chdepth_{R\wslash xR}( M\wslash  x M) = c -1$. 
Thus by the induction hypothesis, 
the maximal length of cohomological $M\wslash  xM$-regular sequence is $c -1$. 
Thus we have $r = c$.  
\end{proof}

\subsubsection{Auslander-Buchsbaum formula} 

\begin{theorem}\label{Auslander-Buchsbaum formula}
Let $R$ be a piecewise Noetherian CDGA with a maximal ideal $\ulfrkm$ and  the residue field $k$. 
Let $M \in \sfD_{\mod H^{0}}(R)_{\textup{fpd}}$. 
Then,
\[
\inf \RHom(k, M) + \pd M - \sup M= \inf \RHom(k, R) . 
\]
If moreover we assume $H^{\ll 0}= 0$, then, 
\[
\chdepth M + \pd M - \chdepth R = \champ M - \champ R.
\]
\end{theorem}

\begin{proof}
First note we have an isomorphism 
\begin{equation}\label{201712021706} 
\RHom(k, M) \cong \RHom(k, R) \lotimes_{R} M. 
\end{equation}

Since $\RHom(k,R)$ belongs to $\sfD(k)$, 
it is a shifted direct sum of $k$. 
Namely, it is of the form 
\[
\RHom(k,R) = \bigoplus_{s=a}^{b}k^{\oplus n_{s}}[-s]
\]
where $a= \inf \RHom(k,R), b =\sup \RHom(k,R)$. 
Since $M$ belongs to $\sfD_{\mod H^{0}}^{< \infty}(R)$, 
it has a minimal sppj resolution $P_{\bullet}$. 
Thus, we can compute 
$\RHom(k,R) \lotimes_{R} M$  
by using Corollary \ref{sppj resolution corollary 2}. 
In particular, we obtain the equation 
\[
\inf (\RHom(k, R) \lotimes_{R} M) = \inf \RHom(k, R) - \pd M + \sup M.
\] 
Hence, from the isomorphism (\ref{201712021706}) we deduce the first equality.

In the case $H^{\ll 0}=0$, we have 
\[
\begin{split}
\chdepth M + \pd M - \chdepth R 
&= \inf\RHom(k, M) -\inf M \\
  & \ \ \ \ \ \ \ \ \ \ \ \ \ + \pd M  -\inf \RHom(k, R)  +\inf R\\ 
& = \champ M - \champ R.
\end{split}
\]
\end{proof}

\begin{remark}
This formula for the case $H^{\ll 0} = 0$ was proved by Frankild-Jorgensen \cite{Frankild-Jorgensen:Homological identities} 
for a nonocommutative DGA having a dualizing complex.
\end{remark}

\begin{remark}\label{Amplitude remark} 
P. Jorgensen \cite{Jorgensen:Amplitude} showed that  
if $H^{\ll 0} =0$, then we have 
\[\champ M - \champ R \geq 0\]  
for $M \in \sfD_{\mod H^{0}}(R)_{\textup{fpd}}, M\neq 0$  
under a slightly stronger assumption. 
\end{remark}

\subsection{Indecomposable injective DG-modules and the Bass numbers}\label{CDGA:Indecomposable}

\subsubsection{Indecomposable injective modules}

For $\frkp \in \Spec H^{0}$, 
we define the DG-$R$-module $E_{R}(R/\ulfrkp')$ to be 
\[
E_{R}(R/\ulfrkp') := \Hom_{R^{0}}^{\bullet}(R, E_{R^{0}}(R^{0}/\ulfrkp)). 
\] 
We note that since $E_{R^{0}}(R^{0}/\ulfrkp) = E_{R^{0}}(E_{H^{0}}(H^{0}/\frkp))$, we have $E_{R}(R/\ulfrkp') = G(E_{H^{0}}(H^{0}/\frkp))$ where $G$ is the functor defined in Section \ref{the class cI}. 
It is clear  that $E_{R}(R/\ulfrkp')$ belongs to $\cI$ 
and hence that under the equivalence $\tuH^{0}: \cI \to \Inj H^{0}$,  
it corresponds to $E_{H^{0}}(H^{0}/\frkp)$. 

We remark that the quasi-isomorphism class of $E_{R}(R/\ulfrkp')$, i.e., 
an object of  $\cI$ corresponding  to $E_{H^{0}}(H^{0}/\frkp) \in \Inj H^{0}$,  
 is already studied by Shaul \cite{Shaul:Injective}.

Since we are assuming $R$ is piecewise Noetherian, 
$\cI$ is closed under direct sums inside $\cD(R)$ by Bass-Papp theorem for DG-algebras \cite{coppepan, Shaul:Injective}.
It follows that the coproduct  of the category $\cI$ 
is the direct sum taken inside $\cD(R)$. 
Similarly, 
the coproduct of the category $\Inj H^{0}$ 
is the direct sum inside $\Mod H^{0}$. 
Therefore, 
using the equivalence $\cI \cong \Inj H^{0}$, 
we can deduce the following statements 
from the same statement for a ordinary commutative ring 
(see e.g. \cite[Theorem 18.4,Theorem 18.5]{Matsumura}).  

\begin{proposition}[{Shaul \cite{Shaul:Injective}}]\label{decomposition of injectives}
The following statements hold. 
\begin{enumerate}[(1)] 
\item An indecomposable object of $\cI$ 
is the form of $E_{R}(R/\ulfrkp')$ for some $\frkp \in \Spec H^{0}$. 

\item 
An object $I \in \cI$ is 
a direct sum of objects of the forms $E_{R}(R/ \ulfrkp')$. 
\end{enumerate}
\end{proposition}

The special case $\frkp = \frkq$ of following lemma  is proved by Shaul \cite{Shaul:Injective}.

\begin{lemma}
For another $\frkq \in \Spec H^{0}$, 
the canonical map $\mathsf{can}$  
below is a quasi-isomorphism
\[
\mathsf{can}: E_{R}(R/\ulfrkp')_{\ulfrkq'}
 \to 
\Hom^{\bullet}_{R_{\ulfrkq'}}(R_{\ulfrkq'}, E_{R^{0}_{\ulfrkq}}(R^{0}_{\ulfrkq}/\ulfrkp_{\ulfrkq}) )=
E_{R_{\ulfrkq'}}(R_{\ulfrkq'}/\ulfrkp'_{\ulfrkq'}).
\]
\end{lemma}

\begin{proof}
Using \cite[Lemma 3.11]{coppepan}, we have the following isomorphisms. 
\[
\begin{split}
&\tuH(E_{R}(R/\ulfrkp')_{\ulfrkq'}) \cong \tuH(E_{R}(R/\ulfrkp'))_{\frkq'} \cong \Hom_{H^{0}}^{\bullet}(H, E_{H^{0}}(H^{0}/\frkp))_{\frkq'} 
\cong E_{H}(H/\frkp')_{\frkq'}, \\ 
&\tuH(E_{R_{\ulfrkq'}}(R_{\ulfrkq'}/\ulfrkp'_{\ulfrkq'})) \cong  \Hom_{H^{0}_{\frkq}}^{\bullet}(H_{\frkq'}, E_{H^{0}_{\frkq}}(H^{0}_{\frkq}/\frkp_{\frkq})) 
\cong E_{H_{\frkq'}}(H_{\frkq'}/\frkp'_{\frkq'}).
\end{split}
\]
It is well-known that there exists a canonical isomorphism between 
$E_{H}(H/\frkp')_{\frkq}$ and $E_{H_{\frkq'}}(H_{\frkq'}/\frkp'_{\frkq'})$, which coincides with $\tuH(\mathsf{can})$ under the above isomorphism. 
\end{proof}

There are  two immediate  consequences.

\begin{corollary}\label{20181162135}
\begin{enumerate}[(1)]
\item 
The localization functor $(-)_{\ulfrkp'}$ sends $\cI(R)$ to 
$\cI(R_{\ulfrkp'})$ and makes 
the following diagram commutative. 
\[
\begin{xymatrix}{ 
\cI(R) \ar@{=}[rr]^{\tuH^{0} \ \sim} \ar[d]_{(-)_{\ulfrkp'}} & &
\Inj H^{0} \ar[d]^{(- )_{\frkp}} \\ 
\cI( R_{\ulfrkp'} ) \ar@{=}[rr]^{\tuH^{0} \ \sim} & &
\Inj H^{0}_{\frkp}  
}\end{xymatrix} 
\]

\item 
$\id_{R_{\ulfrkp'}} M_{\ulfrkp'} \leq \id_{R} M$.

\item 
The localization functor $(-)_{\ulfrkp'}$ preserves finiteness of injective dimension. 
Namely, 
if $M \in \sfD(R)_{\textup{fid}}$, then $M_{\ulfrkp'} \in \sfD(R_{\ulfrkp'})_{\textup{fid}}$.  
\end{enumerate}
\end{corollary}

\begin{proposition}\label{20181162135II}
Let $M \in \sfD_{\mod H^{0}}(R)$ and $N \in \sfD(R)_{\textup{fid}}$. 
Then the canonical morphism below is an isomorphism
\[
\Phi_{M, N} :\RHom_{R}(M,N)_{\ulfrkp} \xrightarrow{ \cong} \RHom_{R_{\ulfrkp}}(M_{\ulfrkp}, N_{\ulfrkp}).
\]
\end{proposition}

\begin{proof}
By divisage argument, we may reduce the problem to the case $N \in \cI$. 
 We have to prove that 
 the morphism below 
 \[
 \phi_{M,N[n]}: \Hom(M,N[n])_{\frkp} \to \Hom(M_{\ulfrkp'}, N_{\ulfrkp'}[n])
 \]
 is an isomorphism for $n \in \ZZ$.
 \[
 \begin{split} 
 \Hom(M,N[n])_{\frkp}
 & \cong \Hom(\tuH^{-n}(M), \tuH^{0}(N))_{\frkp} \\
 & \cong \Hom(\tuH^{-n}(M)_{\frkp}, \tuH^{0}(N)_{\frkp})\\
 & \cong \Hom(\tuH^{-n}(M_{\ulfrkp'}), \tuH^{0}(N_{\ulfrkp'})) 
\cong \Hom(M_{\ulfrkp'}, N_{\ulfrkp'}[n])
\end{split}
\]
\end{proof}

\subsubsection{Bass number}

The Bass number is defined in the same way of usual commutative algebras. 

\begin{definition}[Bass number]
For  $\frkp \in \Spec H^{0}$,  we set 
\[
\mu^{n}(\frkp, M) := 
\dim_{\kappa(\frkp)}
 \Hom_{R_{\ulfrkp}}(
\kappa(\frkp), M_{\frkp}[n]).
\]
\end{definition}

The structure theorem of minimal injective resolution of $M \in \Mod A$
 over an ordinary Noetherian algebra $A$  \cite[Theorem 18.7]{Matsumura} 
holds for a minimal ifij resolution of $M \in \sfD(R)$ 
over  $R$, 
which can be proved by the same way of the ordinary case with the aide of Corollary \ref{ifij resolution corollary 1}. 

\begin{theorem}\label{Bass injective decomposition theorem}
Let $M \in \sfD^{> -\infty}(R)$ and $I_{\bullet}$ a minmal ifij resolution of $M$. 
Then, $\mu^{n}(\frkp, M) = 0 $ if $n \neq i + \inf I_{-i}$ for any $i \in \NN$. 
If for $i \in \NN$ we set $n = i +\inf I_{-i}$,  then
\[
I_{-i} \cong 
\left(\bigoplus_{\frkp \in \Spec H^{0}} E_{R}(R/\ulfrkp')^{\oplus \mu^{ n }(\frkp, M)}\right)[-\inf I_{-i}]. 
\]
\end{theorem}

We collect basic properties of the Bass numbers.

\begin{lemma}\label{pre one up lemma}
For $M \in \sfD^{-\infty}_{\mod H^{0}}(R)$ 
and $i \in \ZZ$, we have 
 $\mu^{n}(\frkp, M) > 0$ 
if and only if $\Hom(H^{0}/\frkp, M[n]) \neq 0$
\end{lemma}

\begin{proof}
The claim follows from the facts that 
\[
\Hom_{R_{\ulfrkp}}(\kappa(\frkp), M_{\frkp}[n]) \cong \Hom_{R}(H^{0}/\frkp, M[n])_{\frkp}\] 
and 
$\Hom_{R}(H^{0}/\frkp, M[n]) $ is a finitely generated $H^{0}/\frkp$-module. 
\end{proof}

The following is essentially a DG-version of \cite[p. 141, Lemma 3]{Matsumura}.

\begin{lemma}\label{one up lemma}
Let $M \in \sfD^{> -\infty}_{\mod H^{0}}(R)$,  
 $i \in\ZZ$ and $\frkp, \frkq \in \Spec H^{0}$. 
We assume that  $\frkp\subsetneq \frkq$ 
and that $\dim H^{0}_{\frkq}/\frkp H^{0}_{\frkq}=1$. 
If $\mu^{n}(\frkp, M) > 0$, then  
$\mu^{n+ 1}(\frkq, M) > 0$. 
\end{lemma}

\begin{proof}
Localizing at $\frkq$, we may assume $\frkq$ is a maximal. 
Take $x \in \frkq \setminus \frkp$ and consider the exact sequence 
\[
0 \to H^{0}/\frkp \xrightarrow{x \times -} H^{0}/ \frkp \to H^{0}/(xH^{0} + \frkp) \to 0. 
\]
Applying $\Hom(-, M)$, we obtain the exact sequence  
\[
\Hom(H^{0}/ \frkp, M[n]) \xrightarrow{x \times -} \Hom(H^{0}/\frkp, M[n]) 
\to  \Hom(H^{0}/(x H^{0} +\frkp), M[n +1]). 
\]
If $\Hom(H^{0}/(x H^{0} +\frkp), M[n +1]) = 0$, then 
$\Hom(H^{0}/\frkp, M[n]) = 0$ by Nakayama's Lemma, a contradiction. 
Thus $\Hom(H^{0}/(x H^{0} +\frkp), M[n +1]) \neq 0$. 
Since $H^{0}/(xH^{0} + \frkp)$ is obtained from  $ H^{0}/\frkq$ by extensions, 
we must have $\Hom(H^{0}/\frkq, M[n +1]) \neq 0$. 
\end{proof}

\begin{corollary}\label{201712082342}
Let  $M \in \sfD_{\mod H^{0}}^{< \infty}(R)$ 
with a minimal ifij resolution $I_{\bullet}$.  
Let $i \in \NN$ and  set $n := i + \inf I_{-i}$. 
Assume that there exist a non-maximal prime ideal 
$\frkp$  
such that $\mu^{ n}(\frkp, M) \neq 0$. 
Then $\inf I_{-i} = \inf I_{ -i -1}$. 
\end{corollary}

Before giving a proof, we recall  the inequality $\inf I_{-i} \leq \inf I_{-i -1}$ from  \eqref{201903141554}.

\begin{proof}
Assume on the contrary that $\inf I_{-i} < \inf I_{-i-1}$. 
Then for $k,j $ such that $k< i < j$, we  have  
\[
k + \inf I_{-k} \leq i + \inf I_{-i}=n,\ \ \  n +1 = i +1 + \inf I_{-i} < j + \inf I_{-j}. 
\]
Therefore, there exists no $j \in \NN$ such that $j+ \inf I_{-j} = n +1$. 
Thus for any $\frkq \in \Spec H^{0}$ we have 
$\mu^{n +1} (\frkq , M) = 0$ by Theorem \ref{Bass injective decomposition theorem}. 
This contradicts to Lemma \ref{one up lemma}. 
\end{proof}

\begin{corollary}\label{bound of dimension lemma}
If $\id M < \infty$, 
then $\dim \tuH^{\inf}(M) \leq \id M$. 
\end{corollary}

\begin{proof}
This follows from the fact that a injective envelope $J^{0}$ 
of $\tuH^{\inf}(M)$ is a direct sum of injective envelope of minimal prime ideals of $M$. 
\end{proof}

The injective dimension $\id M$ can be measured by the Bass numbers $\mu^{n}(\frkm, M)$ 
of  maximal ideals $\frkm$.

\begin{proposition}\label{criterion for finiteness of injective dimension}
For  $d \in \NN$ and   $M \in \sfD^{>-\infty}_{\mod H^{0}}(R)$ with $b= \inf M$,  
the following conditions are equivalent. 
\begin{enumerate}[(1)]
\item 
$\injdim M = d$ 

\item The following conditions are satisfied. 
\begin{enumerate}
\item 
$\mu^{d +b}(\frkm, M) \neq 0$  for some maximal ideal $\frkm$ of $H^{0}$. 
\item 
$\mu^{n}(\frkm, M) = 0$ 
for  $n > d +b$ and any maximal ideal $\frkm$ of $H^{0}$. 
\end{enumerate}
\end{enumerate}
\end{proposition}

We need the following lemma which is  a DG-version of \cite[p 139, Lemma 1]{Matsumura}.

\begin{lemma}\label{201712011651}
Let $M \in \sfD(R)$ and 
 $b:= \inf M$. 
Then for $d \in \ZZ$, 
the following conditions are equivalent: 
\begin{enumerate}[(1)]
\item 
$\injdim M \leq  d$ 
\item 
$\Hom(N , M[c+ 1 + b]) = 0$ 
for $N \in \mod H^{0}$ and for $c \geq d$
\item 
$\Hom(H^{0}/\frkp, M[c +1 +b]) = 0$ for all 
$\frkp \in \Spec H^{0}$ and for $c \geq d$. 
\end{enumerate}
\end{lemma}

\begin{proof}
The implications (1) $\Rightarrow$ (2) $\Rightarrow$ (3) are clear. 
(2) $\Rightarrow$ (1) follows from Theorem \ref{ifij resolution theorem}. 
(3) $\Rightarrow $ (2) follows from the fact that 
every finite $H^{0}$-module  is obtained from $\{ H^{0}/ \frkp\mid \frkp \in \Spec H^{0}\}$ 
by extensions. 
\end{proof}

\begin{proof}[Proof of Proposition \ref{criterion for finiteness of injective dimension}]
We prove the implication (1) $\Rightarrow $ (2). 
The condition (b) follows from the definition of the injective dimension and Lemma \ref{pre one up lemma}. 
By Lemma \ref{201712011651} and Lemma \ref{pre one up lemma}, 
there exists $\frkp \in \Spec H^{0}$ such that $\mu^{d +b}( \frkp, M) \neq 0$. 
However, 
by  Lemma \ref{one up lemma}, such  $\frkp$ must be a maximal ideal.

We prove the implication (2) $\Rightarrow$ (1). 
If there exists $\frkp \in \Spec H^{0}$ 
such that $\mu^{n}(\frkp, M) \neq 0$ for some $n > b+d$, 
then by Lemma \ref{one up lemma} 
there exists $m \geq  n$  and a maximal ideal containing $\frkp$ 
such that $\mu^{m}(\frkm, M) \neq 0$. 
Therefore, by the assumption (b), we must have $\mu^{n}(\frkp, M) = 0$ for $n > b+d, \frkp \in \Spec H^{0}$. 
By Lemma \ref{201712011651} and Lemma \ref{pre one up lemma}, 
we deduce that $\injdim  M \leq d$. 
From the assumption (a), we conclude that $\injdim M =d$. 
\end{proof}

In the case where $R$ is local, we obtain the following corollary. 

\begin{corollary}\label{criterion for finiteness of injective dimension corollary}
Let $R$ be a piecewise Noetherian CDGA with  the residue field $k$.  
Then, for $M\in \sfD^{> -\infty}(R)$ we have 
\[
\id M = \sup \RHom(k, M) - \inf M. 
\]
Thus in particular, $\injdim M < \infty$ 
if and only if 
$\sup \RHom(k, M) < \infty$. 
\end{corollary}

Combining Lemma \ref{amplitude lemma} and Lemma \ref{criterion for finiteness of injective dimension corollary} above, 
we deduce 

\begin{corollary}\label{waru corollary}
Let $R$ be a piecewise Noetherian CDGA with a maximal ideal $\ulfrkm$  and the residue field $k$.  
Let $x \in \ulfrkm$ an  element and $S := R\wslash xR$. 
Then, 
\[
\id_{R} M -1 \leq \id_{S} M\wslash xM \leq \id_{R} M 
\]
 In particular $\injdim_{R} M < \infty$ if and only if $\injdim_{S} M\wslash xM < \infty$. 
Moreover, $\id_{R}M -1 = \id_{S}M \wslash xM$ 
if and only if $x$ is cohomological $M$-regular element.  
\end{corollary}

We prove a DG-version of the Bass formula. 

\begin{theorem}[Bass formula]\label{Bass formula} 
Let $R$ be a piecewise Noetherian CDGA with a maximal ideal $\ulfrkm$  and the residue field $k$.  
Assume that $H^{\ll 0} = 0$. 
If $M \in \sfD^{\mrb}_{\mod H^{0}}(R) $ satisfies $\injdim M < \infty$, then 
\[
\chdepth R -\champ R =  \injdim M - \champ M.
\]
\end{theorem}

\begin{remark}
This formula was proved by Frankild-Jorgensen \cite{Frankild-Jorgensen:Homological identities} 
for a nonocommutative DGA having a dualizing complex.
\end{remark}

\begin{proof}
Let 
$
p := \chdepth R, 
q := \injdim M, 
a =   \inf R, 
b= \inf M,
c := \sup M. 
$
Note that $\champ R = - a, \champ M = c-b$. 
Take a cohomological $R$-regular sequence $x_{1}, \cdots, x_{p}$.  
Then, since the Koszul complex 
$K := K(x_{1}, \cdots, x_{p})$ belongs to $\sfD^{[a, 0]}(R)$, 
we have $\RHom(K ,M) \in \sfD^{[b, q -a +b]}(R)$. 
On the other hand, we have
\[
\begin{split}
&\Hom(K, M[c + p]) = \tuH^{c+p}(\RHom(K, M) ) = \tuH^{c}(M)/ (x_{1}, \cdots, x_{p}) \tuH^{c}(M) \neq 0, \\
&\Hom(K, M[i ]) = \tuH^{i}(\RHom(K, M) ) = \tuH^{ i -p}(M)/ (x_{1}, \cdots, x_{p})\tuH^{i-p}(M) = 0 \textup{ for } i > c +q
\end{split}  
\]
Consequently, $ p + c \leq q -a + b$. 

Assume that $ p + c < q - a +b$. 
Let  $f: k [-a] \to  K$ a nonzero morphism. 
Then, 
 the induced morphism $k \to \tuH^{a}(K)$  is injective. 
We set $L := \cone (f)$.  
Applying $\Hom(-, M)$ to the exact triangle $k[-a] \to K \to L$ 
we obtain the exact sequence 
\[
\Hom(K,M[q -a + b]) \to \Hom(k[-a], M [q-a +b]) \to \Hom(L, M [q-a +b + 1] ).  
\]
It is already shown that the left term is zero. 
Since  the cone $L = \cone (f)$ belongs to $\sfD^{[a, 0]}(R)$, 
$\RHom(L,M)$ belongs  to $\sfD^{[b, q + b -a]}$. 
Therefore the right term is zero. 
Consequently, looking the middle term, 
we come to the contradiction $\Hom(k, M [q +b]) = 0$.
Thus we must have $p + c = q -a + b$ 
or equivalently $p -a = q -(c-b)$. \color{black}
\end{proof}

\begin{corollary}
If $\injdim R < \infty$, then 
\[
\dim \tuH^{\inf}(R) \leq \injdim R = \chdepth R. 
\] 
\end{corollary}

From the inequality given in Remark \ref{Amplitude remark}, 
it is natural to expect the same inequality holds 
for an object $M \in \sfD_{\mod H^{0}}(R)_{\textup{fid}}$. 
Shaul informed the author  in private communication that he also conjectured the same statement.

\begin{conjecture}
For $M \in \sfD_{\mod H^{0}}(R)_{\textup{fid}}$, 
\[
\champ M \geq \champ R. 
\]
\end{conjecture}

\begin{remark}
This conjecture is solved affirmatively by Shaul \cite{Shaul CM}.
\end{remark}

\section{A dualizing complex and a Gorenstein CDGA}\label{CDGA:dualizing}

\subsection{Dualizing complexes and the structure of their minimal ifij resolution}

\subsubsection{Dualizing complexes}

We follow  Yekutieli's definition of a dualizing complex of CDGA. 
We refer \cite{Yekutieli} for other definitions of a dualizing complex  and comparisons to this definition.

\begin{definition}[{Yekutieli}]\label{dualizing complex definition}
An object $D \in \sfD(R)$ is called \textit{dualizing}, 
if it satisfies the following conditions. 
\begin{enumerate}[(1)]
\item $D \in \sfD_{\mod H^{0}}(R)_{\textup{fid}}$, 
i.e., 
$\injdim D < \infty$ and $\tuH^{n}(D)$ is finitely generated $H^{0}$-module 
for $n \in \ZZ$. 

\item 
The homothety morphism $R \to \RHom(D,D)$ is an isomorphism 
in $\sfD(R)$.
\end{enumerate}
\end{definition}

Combining Corollary \ref{20181162135} and Proposition \ref{20181162135II} 
we obtain the following corollary. 

\begin{corollary}\label{201712012051}
If $D \in \sfD(R)$ is a dualizing object, then so is the localization  $D_{\ulfrkp'} \in \sfD(R_{\ulfrkp'})$. 
\end{corollary}

We give a DG-version of \cite[V. Proposition 3.4]{Residues and Duality}.
\begin{theorem}\label{characterization of a dualizing complex}
Let $R$ be a piecewise Noetherian CDGA with a maximal ideal $\ulfrkm$  and the residue field $k$.  
Then $D \in \sfD^{> -\infty}_{\mod H^{0}}(R)$ 
is dualizing 
if and only if 
$\RHom(k, D) \cong k[n]$ 
for some $n$. 
\end{theorem}

\begin{remark}
This result was proved in the case $H^{\ll 0}= 0$ 
by Frankild Jorgensen and Iyengar \cite[Theorem 3.1]{FIJ}.
\end{remark}

\begin{proof}
We prove ``if" part. 
By Proposition \ref{criterion for finiteness of injective dimension}, 
$D$ has finite injective dimension. 
We set $ F  := \RHom(-,D)$.  
It is remained to prove that the canonical evaluation morphism $R \to F^{2}(R)$ 
is an isomorphism.
For this we prove that 
the natural morphism $\eta_{M} : M \to F^{2}(M)$ 
is an isomorphism for $M \in \sfD^{< \infty}_{\mod H^{0}}(R)$. 

By the same argument with \cite[V. Proposition 2.5]{Residues and Duality}, 
we can show  that $\eta_{M}$ is an isomorphism 
for $M \in \mod H^{0}$. 
Therefore by divisage argument, 
we can show that $\eta_{M}$ is an isomorphism for $M \in \sfD^{\mrb}_{\mod H^{0}}(R)$.

We deal with the general case. 
Since $\sfD^{< \infty}_{\mod H^{0}}(R)$ is stable under the shift $[n]$ for $n \in \ZZ$, 
it is enough to show that $\tuH^{0}(\eta_{M})$ is an isomorphism. 
Let $d:= \injdim M $ and $b:= \inf M$. 
Then $F(\sfD^{[m,n]}(R)) \subset \sfD^{[b-n, b+d -m]}(R)$. 
Therefore, 
\[
F^{2}(\sfD^{< -d}(R)) \subset F(\sfD^{> d+b}(R)) \subset \sfD^{< 0}(R).
\]
Consequently, we have 
$\tuH^{0}(F^{2}(\sigma^{< -d}(M))) = 0, \tuH^{1}(F^{2}(\sigma^{< -d}(M))) = 0$. 
Looking the cohomology long exact sequence of the exact triangle 
\[
F^{2}(\sigma^{< -d}(M)) \to F^{2}(M) \to F(\sigma^{\geq d}(M)) \to 
\]
we see that  the canonical morphism 
$\tuH^{0}(F^{2}(M)) \to \tuH^{0}(F^{2}\sigma^{\geq -d}(M))$ is an isomorphism. 
Similarly, 
since $\tuH^{0}(\sigma^{<-d}(M)) = 0$ and $\tuH^{1}(\sigma^{< -d} (M)) = 0$, 
we see that 
the canonical morphism $\tuH^{0}(M) \to \tuH^{0}(\sigma^{\geq -d}(M))$ is an isomorphism.

Since $\eta$ is a triangulated natural transformation, 
we have the following commutative diagram whose both lows are exact. 
\[
\begin{xymatrix}{ 
\tuH^{0}(M) \ar[r]^{\cong} \ar[d]_{\tuH^{0}(\eta_{M})} & 
\tuH^{0}(\sigma^{\geq -d} M)  \ar[d]^{\tuH^{0}(\eta_{\sigma^{\geq -d}M})} \\ 
\tuH^{0}(F^{2}M) \ar[r]^{\cong}  & 
\tuH^{0}(F^{2}\sigma^{\geq -d} M)  
}\end{xymatrix}
\]
By induction step, the right vertical arrow is an isomorphism, 
so is the left.  This completes the proof of ``if" part. 

``Only if" part can be proved by a similar argument to 
the proof of the same statement for a dualizing complex over a ordinary commutative algebra. 
\end{proof}

\subsubsection{The structure of minimal ifij resolution of a dualizing complex}
We establish a structure theorem of a minimal ifij resolution of a dualizing complex. 

\begin{theorem}\label{structure of minamal ifij resolution of a daulizing complex}
If $R$ has a dualizing complex  $D \in \sfD(R)$ with a minimal ifij resolution $I_{\bullet}$ of length $e$.  
Then  $H^{0}$ is catenary and $\dim H^{0} < \infty$. 
If moreover we assume that $R$ is local, 
then  the following statements hold. 
\begin{enumerate}[(1)]

\item  $\inf I_{-i} = \inf D$ for $i = 0, \cdots, e$.  

\item $e = \injdim D =\depth D = \dim H^{0}$.

\item 
\[
I_{-i} = \bigoplus_{\frkp } E_{R}(R/\ulfrkp')[-\inf D].  
\]
where $\frkp$ run all prime ideals such that $i = \dim H^{0} - \dim H^{0}/\frkp$. 
 \end{enumerate}
 \end{theorem}

\begin{proof}
Combining Corollary \ref{201712012051} and Theorem \ref{characterization of a dualizing complex}, 
we see that 
for $\frkp \in \Spec H^{0}$, 
there exists $n_{\frkp} \in \ZZ$ 
such that $\mu^{m}(\frkp, D) =  \delta_{m,n_{\frkp}}$. 
By Lemma \ref{one up lemma}, 
we have $n_{\frkp} = n_{\frkq} -1$ 
for a pair $\frkp \subset \frkq$ satisfying $\dim H^{0}_{\frkq}/\frkp H^{0}_{\frkq} = 1$.  
It follows that $H^{0}$ is catenary (see \cite[V. 7.2]{Residues and Duality}).
Moreover since the set $\{ n_{\frkp} \mid \frkp \in \Spec H^{0} \}$ is finite, 
we conclude that $\dim H^{0} < \infty$.

From now we assume that $H^{0}$ is local with a maximal ideal $\frkm$.  
By Lemma \ref{one up lemma}, we have 
$I_{-e} = E_{R}(R/\ulfrkm')[-\inf I_{e}]$. 
For $i < e$, there exists a non-maximal prime ideal $\frkp$ 
such that $n_{\frkp} = i+ \inf I_{-i}$ by Lemma \ref{one up lemma}.  
Thus by Corollary \ref{201712082342}, we deduce the condition  (1). 
Now it is easy to deduce (3) by descending induction on $i$.
It follows $\dim H^{0} =e$.
The equation $\id D = e$   
is a consequence of  Proposition \ref{criterion for finiteness of injective dimension}.
Finally the equality $\depth D = e$ is deduced from a straightforward computation.
\end{proof}

\begin{remark}
An ordinary commutative ring $A$ having a dualizing complex is universal catenary by \cite[V. Section 10]{Residues and Duality}. 
On the other hand, 
by \cite[Proposition 7.5]{Yekutieli}, if $R$ has a dualizing complex $D$, 
then $\RHom_{R}(H^{0}, D)$ is a dualizing complex over $H^{0}$. 
It follows  that if $R$ has a dualizing complex $D$, 
then $H^{0}$ is universal catenary. 
The author thanks L. Shaul for pointing out this statement. 
\end{remark}

 \subsubsection{When is the cohomology module $\tuH(D)$  a dualizing complex of $H$.}
We discuss a  condition that the cohomology group $\tuH(D)$ of  a dualizing complex  $D$ 
is a dualizing complex over $H$ where we regard $\tuH(D)$ as a DG-$H$-module with the trivial differential.

\begin{theorem}\label{dualizing formal theorem}
Let $R$ be a piecewise Noetherian CDGA.   Set $ d: = \dim H^{0}$. 
For $D \in \sfD^{> -\infty}_{\mod H^{0}}(R)$ with $b= \inf D$, 
the following conditions are equivalent.
\begin{enumerate}[(1)]
\item 
$D$  is a dualizing complex 
such that
the cohomology groups of  a minimal ifij resolution 
\[
D \to I_{0} \to I_{-1} \to \cdots \to I_{-d}
\]
gives an injective resolution 
\begin{equation}\label{201711241918}
0\to \tuH(D) \to  \tuH(I_{0}) \to  \tuH(I_{-1}) \to \cdots \to  \tuH(I_{-d}) \to 0. 
\end{equation}

\item 
The following conditions are satisfied.
\begin{enumerate}[(a)]
\item 
$C= H^{b}(D)$ is  a canonical module over $H^{0}$, 
that is, a dualizing complex concentrated at degree $0$. 

\item  
$H^{-n}$ is a maximal Cohen-Macaulay (MCM)-module for $n \geq 0$. 

\item 
There exists an isomorphism 
\[
\psi: \tuH(D) \to \Hom_{H^{0}}^{\bullet}(H, C)[-b]
\]  of graded $H$-modules. 
\end{enumerate}
\end{enumerate}
\end{theorem}

\begin{proof} 
We prove the implication (1) $\Rightarrow$ (2). 
Since $b= \inf I_{-i}$ for $i = 0, \dots, d$ by Theorem \ref{structure of minamal ifij resolution of a daulizing complex}, 
looking at the $b$-th degree of the  exact sequence (\ref{201711241918}), 
we obtain  a minimal injective resolution of $C := \tuH^{b}(D)$  
\begin{equation}\label{201712021619}
0 \to C \to K_{0} \to K_{-1} \to \cdots \to K_{-d} \to 0 
\end{equation}
where we set $K_{-i} := \tuH^{b}(I_{-i})$. 
Since $\tuH^{0}(E_{R}(R/\ulfrkp')) =E_{H^{0}}(H^{0}/\frkp)$, 
combining Theorem \ref{Bass injective decomposition theorem} and 
 Theorem \ref{characterization of a dualizing complex}, 
 we see that $\mu_{H}^{n}(\frkm, C) = \delta_{n,d}$. 
 Thus, the module $C$ is a dualizing complex by Theorem \ref{characterization of a dualizing complex}. 
 Since $\tuH(I_{-i}) = \Hom_{H^{0}}^{\bullet}(H ,K_{-i})$, 
looking at the $n$-th degree of the exact sequence (\ref{201711241918}), 
we see that 
the complex 
$\Hom(H^{-n}, K_{\bullet})$ below is exact. 
\[
\Hom(H^{n}, K_{\bullet}) : \Hom(H^{n}, K_{0} ) \to \Hom(H^{n}, K_{-1}) \to \cdots \to \Hom(H^{n}, K_{-d}) \to 0 
\]
In other words, $\Ext^{i}_{H^{0}}(H^{-n}, C) = 0$ for $i >0$. 
It is well-know  (e.g. \cite[Proposirtion 3.3.3]{Bruns-Herzog})
that this implies  that  $H^{-n}$ is a MCM-$H^{0}$-module. 
Finally, from the canonical isomorphism $\tuH(I_{-i}) \cong \Hom_{H^{0}}(H, K_{-i})[-b]$ 
we deduce (c) as below 
\[
\begin{split}
\tuH(D) = \Ker[\tuH(I_{0}) \to \tuH(I_{-1}) ] 
&\cong 
\Ker\left[\Hom_{H^{0}}^{\bullet}(H,K_{0})[-b] \to \Hom_{H^{0}}^{\bullet}(H, K_{-1})[-b] \right]\\
& = \Hom_{H^{0}}^{\bullet}(H, C)[-b].
\end{split}
\]

We prove the implication (2) $\Rightarrow$ (1). 
We use a result obtained in Appendix \ref{CE injective resolution}. 
Let $K_{\bullet}$ be a minimal injective resolution of $C$.    
Then the collection  of graded $H$-modules  $J_{i} := \Hom_{H^{0}}^{\bullet}(H, K_{i})$ 
forms a complex $J_{\bullet}$ of graded $H$-modules.  
Since by the assumption $\Ext_{H^{0}}^{i}(H^{-n}, C) = 0$ for $ i > 0, n \geq 0$, 
 the complex  $J_{\bullet }$  is exact. 
 Moreover since $\Ker[J_{0} \to J_{-1} ] = \Hom_{H^{0}}(H,C)$ is isomorphic to $\tuH(D)$ by the  assumption, 
 we obtain an exact sequence of graded $H$-modules 
\[
0\to \tuH(D) \xrightarrow{\overline{\delta}_{1}}  J_{0} 
\xrightarrow{\overline{\delta}_{0}} J_{-1} 
\xrightarrow{\overline{\delta}_{-1}} \cdots 
\xrightarrow{\overline{\delta}_{-d+1}} J_{-d} \to 0. 
\]

By Appendix \ref{CE injective resolution} 
there exists an exact sequence inside $\sfC(R)$ 
\[
0\to D \xrightarrow{\delta_{1}}  I_{0} 
\xrightarrow{\delta_{0}} I_{-1} 
\xrightarrow{\delta_{-1}} \cdots 
\xrightarrow{\delta_{-d+1}} I_{-d} \to 0 
\]
such that $\tuH(I_{-i}) \cong J_{-i}$. 
It is easy to see that this gives a minimal ifij resolution of $D$. 
From the equation $\mu_{H^{0}}^{i}(\frkm, C) = \delta_{id}$ (see e.g. \cite[Theorem 3.3.10]{Bruns-Herzog}), 
we deduce $\mu_{R}^{i}(\frkm, D) = \delta_{id}$ by the construction. 
Thus, by  Theorem \ref{characterization of a dualizing complex}, we conclude that $D$ is a dualizing complex.  
\end{proof}

\subsection{Gorenstein CDGA}

We recall the definition of a Gorenstein CDGA. 

\begin{definition}
A CDGA $R$ is called Gorenstein if $\injdim R < \infty$. 
\end{definition}

We note that a Gorenstein CDGA $R$ is lower bounded by Lemma \ref{201903152344}. 
It is easy to see that a CDGA $R$ is Gorenstein if and only if $R$ is a dualizing complex over $R$. 
Therefore, combining Proposition \ref{criterion for finiteness of injective dimension corollary}, 
Theorem \ref{characterization of a dualizing complex} and 
Theorem \ref{structure of minamal ifij resolution of a daulizing complex},
 we deduced the following result.

\begin{theorem}\label{Bass theorem}
Let $R$ be a piecewise Noetherian CDGA with a maximal ideal $\ulfrkm$ and the residue field $k$.  
Then the following conditions are equivalent
\begin{enumerate}[(1)]
\item $R$ is Gorenstein. 

\item 
$\RHom(k, R) \cong k[n]$ for some $n \in \ZZ$. 

\item 
$H^{\ll 0} = 0$ and 
$\RHom(k, R) \cong k[- \dim H^{0} - \inf R]$.

\item 
$\dim H^{0} = \injdim R$. 

\end{enumerate}
\end{theorem}

\begin{remark}
The equivalence (1) $\Leftrightarrow$ (2) is proved under the assumption $H^{\ll 0}= 0$ in \cite[Theorem 4.3]{FIJ}. 
In the above theorem we showed that the condition (1) and (2) imply $H^{\ll 0} = 0$. 
\end{remark}

We discuss when a Gorenstein CDGA $R$ has a Gorenstein cohomology algebras $H$. 

\begin{theorem}\label{cohomology Gorenstein theorem}
Let $R$ be a local piecewise Noetherian  CDGA  
such that $H^{\ll 0} = 0$. 
Then 
the following conditions are equivalent. 
\begin{enumerate}[(1)]
\item 
$R$ is Gorenstein CDGA and $H$ is CM as a graded commutative algebra.

\item 
$H$ is Gorenstein CDGA when it is regarded as CDGA with the trivial differential $\partial_{H}= 0$. 

\item 
$H$ is Gorenstein as an ordinary graded commutative algebra. 

\item 
The following conditions are satisfied.
\begin{enumerate}
\item $H^{\inf R}$ is a canonical $H^{0}$-module. 

\item $H^{-n}$ is a MCM-module over $H^{0}$. 

\item 
There exists an isomorphism 
$H \xrightarrow{\cong}  \Hom_{H^{0}}(H,H^{\inf R})$ of graded $H$-modules.
\end{enumerate}

\end{enumerate}
\end{theorem}

\begin{proof}
We  prove the implication (1) $\Rightarrow$ (3). 
We proceed the proof by induction on $r= \dim H = \depth H$. 

In the case where $r = 0$, 
 $R$ belongs to $\cI$. 
 Therefore the cohomology algebra $H$ is self-injective. 
 Hence, in particular Gorenstein. 

Assume that the case $r -1$ is already proved. 
Let  $x$ be a homogeneous $H$-regular element. 
We note that $x \in H^{0}$.  
We denote their lifts by the same symbol $x \in R^{0}$. 
Then, we have $\tuH(R\wslash  xR) = H/(x)$. 
Hence in particular $R\wslash  xR$ is a Gorenstein CDGA such that 
 $\tuH(R\wslash  xR)$ is CM of dimension $r -1$. 
Thus by induction hypothesis, $H/(x)$ is Gorenstein.   
Therefore $H$ is Gorenstein.

The implication (3) $\Rightarrow$ (1) follows from Proposition \ref{201711232212}. 

Now, we have proved the equivalence (1) $\Leftrightarrow$ (3). 
It follows the equivalence (2) $\Leftrightarrow$ (3), since $\tuH(H) \cong H$.

The equivalence (1) $\Leftrightarrow$ (4) immediately follows from Theorem \ref{dualizing formal theorem}. 
\end{proof}

\begin{remark}
The equivalence (2) $\Leftrightarrow$ (3) is proved in \cite{anodai}.
\end{remark}

Since a commutative algebra of dimension $0$ is CM, 
we deduce the following corollary. 

\begin{corollary}[{\cite[Theorem 3.1]{Avramov-Foxby:Locally Gorenstein homomorphism}, \cite[Theorem 3.6]{FHT:Gorenstein spaces}}]
Let $R$ be a CDGA. Assume that $\dim H= 0$. 
Then $R$ is Gorenstein if and only if $\tuH$ is Gorenstein. 
\end{corollary}

The following theorem immediately follows from Corollary \ref{waru corollary}. 
The case where $R$ is an ordinary algebra is proved by Frankild and Jorgensen \cite[Theorem 4.9]{Frankild-Jorgensen:Gorenstein}.

\begin{theorem}\label{FJ theorem}
Let $R$ be a piecewise Noetherian CDGA with a maximal ideal $\ulfrkm$ such that $H^{\gg 0} = 0$.
Let  $x_{1}, \cdots, x_{r} \in \ulfrkm$ 
elements. 
Then, $R$ is Gorenstein if and only if the Kosuzl complex 
$K(x_{1}, \cdots, x_{r})$ is Gorenstein.
\end{theorem}

We give a way to proved examples of Gorenstein DG-algebra $R$ 
such that $\tuH(R)$ is not Gorenstein. 

\begin{example}
Let $A$ be a local Gorenstein algebra, 
$x_{1}, \cdots, x_{r}$ elements.   
Then the Koszul complexe $R := K_{A}(x_{1}, \cdots, x_{r})$ is Gorenstein DG-algebra by Theorem \ref{FJ theorem}. 
If $H = \tuH(R)$ is Gorenstein, 
then it follows from the inequality 
 $\depth H \leq \depth H^{0}$ 
 that $H^{0}$ is Gorenstein. 
Since 
$H^{0} \cong A/(x_{1}, \cdots, x_{r})$, 
in the case where $A/(x_{1}, \cdots, x_{r})$ is not Gorenstein, 
$R$ is a Gorenstein DG-algebra such that $\tuH(R)$ is not Gorenstein. 
\end{example}

\appendix
\section{DG-projective and DG-injective resolutions}\label{CE resolution}

In this Section we review DG-projective and DG-injective resolutions of a DG-$R$-module $M$, 
which are also called  cofibrant and fibrant replacements, semi-projective and semi-injective resolutions.
For the details we refer to  \cite{Keller:ddc}, \cite{Positselski}. 

\subsection{DG-projective resolution}\label{Cartan-Eilenberg projective resolution}

\begin{definition}
Let $M$ be a DG-$R$-module. 
\begin{enumerate}
\item 
$M$ is called \textit{homotopically projective} 
if the complex $\Hom_{R}^{\bullet}(M,A)$ is acyclic for any acyclic DG-$R$-module $A$. 

\item 
$M$ is called \textit{$\#$-projective} 
if $M^{\#}$ is a graded projective $R^{\#}$-module. 

\item 
 $M$ is called \textit{DG-projective} 
if it is homotopically projective and $\#$-projective.
\end{enumerate}
\end{definition}

By $\sfK_{\DGproj}(R) \subset \sfK(R)$ we denote  the full subcategory of DG-projective DG-$R$-modules 
and 
by $\sfK_{\ac}(R) \subset \sfK(R)$ we denote  the full subcategory of acyclic DG-$R$-modules.   
We recall that the derived  category $\sfD(R)$ is defined as  the Verdier quotient $\sfD(R) := \sfK(R)/\sfK_{\ac}(R)$. 
Let $\mathsf{qt}: \sfK(R) \to \sfD(R)$ be the quotient functor.

\begin{theorem}[{\cite[3.1]{Keller:ddc},\cite[1.4]{Positselski}}]\label{DG-projective resolution theorem}
\begin{enumerate}[(1)]

\item 
The  category  
$\sfK(R)$ has the semi-orthogonal decomposition 
\[
\sfK(R) = \sfK_{\DGproj}(R) \perp \sfK_{\ac}(R)
\]

\item 
For $P \in \sfK_{\DGproj}(R)$ and $M \in \sfK(R)$, 
the induced morphism below is an isomorphism 
\[
\mathsf{qt}_{P,M} : \Hom_{\sfK(R)}(P,M) \xrightarrow{\cong}
 \Hom_{\sfD(R)}(P, M)
\]
where on the right hand side, $\mathsf{qt}$ is suppressed. 

\item 
The restriction of $\mathsf{qt}$ to $\sfK_{\DGproj}(R)$ induces an equivalence 
$\sfK_{\DGproj}(R) \to \sfD(R)$ of triangulated categories.  
\end{enumerate}
\end{theorem}

We can deduce 
(2) and (3) from (1)  by formal argument of triangulated categories (e.g. \cite[1.3]{Positselski}).

The main point of proof of (1) is the following  construction. 
Let $M$ be a DG-$R$-module. 
Then there exists    an exact sequence in $\sfC(R)$ 
\[
0 \to K \to P \xrightarrow{f} M \to 0
\]
such that $P$ is DG-projective and $K$ is acyclic. 
More precisely, we can construct a quasi-isomorphism $f: P \to M$ with $P$ DG-projective 
such that $f^{n} : P^{n} \to M^{n}$ is surjective for $n \in\ZZ$ 
and that $K = \Ker f$ is acyclic. 
We explain how to construct such a morphism.

First 
we point out that 
for $P \in \Add R \subset \sfD(R)$ and $M \in \sfD(R)$,  
we have a canonical surjection 
\[
\Hom(P, M) \twoheadrightarrow \Hom_{\GrMod H}(\tuH(P), \tuH(M)).
\]
In other words, 
any morphism $g: \tuH(P) \to \tuH(M)$ has a lift $\underline{g}: P \to M$, 
that is, a morphism such that $\tuH(\underline{g}) = g$.

Key is the following lemma. 
\begin{lemma}\label{20171124145} 
Let $M \in \sfD(R)$ 
and $g: Q \to \tuH(M)$ be a surjective morphism in $\GrMod H$ 
with $Q \in \GrProj H$. 
Let $P' \in \Add R$ be such that $\tuH(P') = Q$. 
Then there exists a $\#$-projective contractible DG-$R$-module $P'' \in \sfC(R)$ 
and a morphism $\underline{g} : P' \oplus P'' \to M$ 
in $\sfC(R)$ 
such that 
$\tuH(\underline{g}) =g$ and that 
each component  $\underline{g}^{n}: (P')^{n} \oplus (P'')^{n} \to M^{n}$ is surjective. 
\end{lemma}

Let $M$ be a DG-$R$-module. 
We take a projective resolution $Q_{\bullet} $ of $\tuH(M)$ inside $\GrMod H$. 
\begin{equation}\label{20171124130}
 \cdots \to Q_{i} \to Q_{i-1} \to \cdots \to Q_{1} \to Q_{0} \to \tuH(M) \to 0. 
\end{equation}
By the above remark, 
there exist $P'_{i} \in \Add R$ such that $\tuH(P'_{i}) = Q_{i}$.

Using Lemma \ref{20171124145}
we can inductively 
construct 
a contractible DG-$R$-module $P''_{i}$ 
which is $\#$-projective for $i \geq 0$ 
and 
an exact sequence inside $\sfC(R)$ 
\[
 P_{i} \to P_{i-1} \to \cdots \to P_{1} \to P_{0} \to M \to 0
\]
where $P_{i} :=P'_{i} \oplus P''_{i}$.

Repeating the process, we obtain a complex $P_{\bullet}$ of objects of $\sfC(R)$. 
\[
 \cdots \to P_{i} \to P_{i-1} \to \cdots \to P_{1} \to P_{0} 
\]

We set $P := \tot P_{\bullet} $ to be the totalization of the complex $P_{\bullet}$. 
Then the canonical morphism  $f: P \to M$ satisfies the desired properties. 

\begin{lemma}\label{20171124151}
Let $M_{i+1} := \Ker[P_{i} \to P_{i-1}]$ 
and $T_{i}$ denotes the totalizatin of the complex 
\[
 P_{i} \to P_{i-1} \to \cdots \to P_{1} \to P_{0}. 
\]
Then the canonical morphism $T_{i} \to M$ fits into the exact triangle 
$M_{i+1}[i] \to T_{i} \to M$ in $\sfD(R)$. 
\end{lemma}

Until now, we don't need to assume  that $H^{>0} = 0$. 
From now in this section, 
we assume that $H^{> 0}= 0$ as in the main body of the paper. 

Let $M \in \sfD^{<\infty}(R)$. 
Then we may take a projective resolution (\ref{20171124130})  
to be such that $\sup Q_{i} \leq \sup M$ for $i\geq 0$. 
Moreover, by the construction 
there exists an exact sequence 
$0 \to \tuH(M_{i+1} ) \to \tuH(P_{i}) \to \tuH(M_{i}) \to 0$ 
for $i \geq 0$ 
where we set $M_{0} := M$. 
Thus inductively we conclude that $\sup M_{i+1} \leq \sup M$. 
Therefore in the above lemma, we have $\sup M_{i+1}[i] \leq \sup M -i$. 
This property is used for the argument of way-out functors in the main body.

Let $M$ be a graded $H$-module. 
We denote by $\grpd_{H} M $ the projective dimension as graded $H$-module. 
In other words, $\grinjdim_{H} M$ is the projective dimension as an object of the abelian category $\GrMod H$. 

For $M \in \sfD(R)$, $\pd_{R} M$ and $\grpd_{H} \tuH(M)$ relate in the following way.

\begin{proposition}\label{CE bounded proposition}
Assume that $H^{\ll 0} = 0$. 
Let $M \in \sfD^{\mrb}(R)$. 
If $\grpd_{H} \tuH(M)< \infty$, 
then $\pd_{R} M < \infty$. 
\end{proposition}

\begin{proof}
By the assumptions, $\tuH(M)$ has a graded projective resolution $Q_{\bullet}$ of finite length 
such that 
there exist integers $a \leq b$ such that 
each term $Q_{i}$ belongs to $\Add \{ H[c] \mid a \leq c \leq b \}$. 
\[
 0\to Q_{e} \to Q_{e-1} \to \cdots \to Q_{1} \to Q_{0} \to \tuH(M) \to 0
\]
Using DG-projective  resolution, we see that 
$M$ is obtained from $e$ objects of $\Add \{ R[c] \mid a \leq c \leq b \}$
 by taking cones $e$-times. 
 This shows that $M$ belongs to $\thick \cP = \sfD(R)_{\textup{fpd}}$. 
\end{proof}

We remark that the converse of above proposition false.

\begin{example}\label{hanrei}
Let $K$ be a field, $R = K[X]/(X^{2})$ the algebra of dual numbers
and $M$ be a DG-$R$-module defined as below 
\[
M :  \cdots \to 0\to R \xrightarrow{X} R \to 0 \to \cdots 
\]
where the right $R$ is  in the $0$-th degree. 

Since $\tuH^{i}(M) =K \ ( i = 0,-1),  =0 \ (i \neq 0,1)$, 
we have $\pd_{H}\tuH(M) = \infty$. 
However, $\pd_{R} M = 1$. 
Actually we have a sppj resolution 
\[
R \xrightarrow{X} R \to M.
\] 
\end{example}

\subsection{DG-injective resolution}\label{CE injective resolution}

\begin{definition}\label{DG-injective resolution definition}
\begin{enumerate}
\item 
A DG-$R$-module $M$ is called \textit{homotopically injective} 
if the complex $\Hom_{R}^{\bullet}(A, M)$ is acyclic for any acyclic DG-$R$-module $A$. 

\item 
A DG-$R$-module $M$ is called \textit{$\#$-projective} 
if $M^{\#}$ is a graded injective $R^{\#}$-module.

\item 
A DG-$R$-module $M$ is called \textit{DG-injective} 
if it is homotopically injective and the underlying graded $R^{\#}$-module is injective.
\end{enumerate}
\end{definition}

The symbol  $\sfK_{\DGinj}(R) \subset \sfK(R)$  denotes  the full subcategory of DG-injective DG-$R$-modules.

\begin{theorem}[{\cite[3.2]{Keller:ddc},\cite[1.5]{Positselski}}]\label{DG-injective resolution theorem}
\begin{enumerate}[(1)]
\item 
The category  
$\sfK(R)$ has the semi-orthogonal decomposition 
\[
\sfK(R) = \sfK_{\ac}(R) \perp \sfK_{\DGinj}(R)
\]

\item 
For $I \in \sfK_{\DGinj}(R)$ and $M \in \sfK(R)$, 
the induced morphism below is an isomorphism 
\[
\mathsf{qt}_{M,I} : \Hom_{\sfK(R)}(M,I) \xrightarrow{\cong}
 \Hom_{\sfD(R)}(M,I)
\]
where on the right hand side, $\mathsf{qt}$ is suppressed. 

\item 
The restriction of $\mathsf{qt}$ to $\sfK_{\DGproj}(R)$ induces an equivalence 
$\sfK_{\DGinj}(R) \to \sfD(R)$ of triangulated categories.  
\end{enumerate}
\end{theorem}

The poof is analogues to that of Theorem \ref{DG-projective resolution theorem}.
The main point is the construction that 
for $M \in \sfC(R)$, 
there exists an exact sequence in $\sfC(R)$ 
\[
0\to M \to I \to C \to 0
\] 
such that $I$ is DG-projective and $C$ is acyclic.

We can construct such a morphism as in the same way of DG-projective resolution. 
In the construction, the DG-$R$-module $R^{\ast} =\Hom_{\ZZ}^{\bullet}(R,\QQ/\ZZ)$ 
plays a role of $R$ for the construction of DG-projective resolution. 
A  key  property of $R^{\ast}$ is  that the map associated to  the cohomology functor $\tuH$ induces 
is surjective 
\begin{equation}\label{201903151738}
\Hom_{\sfC(R)}(M, J) \twoheadrightarrow \Hom_{\GrMod H}(\tuH(M), \tuH(J)). 
\end{equation}
for any shifted  product $I = \prod_{\lambda \in \Lambda} R^{\ast}[n_{\lambda}]$ of $R^{\ast}$.
In other words, 
any morphism $g: \tuH(M) \to \tuH(J)$ has a lift $\underline{g}: M \to J$.

From now 
 we deal with the DG-algebra such that $R^{> 0}$ as in the main body of the paper.  
We denote by  $\Inj^{0} R \subset \sfC(R)$  
the full subcategories consisting of DG-$R$-modules $G(K)$
for  $K \in \Inj H^{0}$ 
where $G(K)$ is the DG-module defined in Section \ref{the class cI} 
 
 The following two lemmas show that 
 as far as a DG-$R$-module $M$ such that $\inf M > -\infty$ concern, 
we may use objects of $\Inj^{0} R$ in place of shifted direct product of $R^{\ast}$ in a construction of DG-injective resolution of $M$.

\begin{lemma}
If $J$ is a shifted product of the objects of $\Inj^{0} R$, then 
the cohomology functor $\tuH$ induces 
a surjection 
\[
\Hom_{\sfC(R)}(M, J) \twoheadrightarrow  \Hom_{\GrMod H}(\tuH(M), \tuH(J)). 
\]
\end{lemma}

\begin{proof}
Taking the homotopy class of a cochain map gives a surjection 
$\Hom_{\sfC(R)}(M, J) \twoheadrightarrow \Hom_{\sfK(R)}(M ,J)$. 
We notice that $G(K) = \psi_{R}(E_{R^{0}}(K))$ where $\psi_{R}$ is defined in \cite[Section 3.2]{coppepan}.  
Now it follows from \cite[Lemma 3.11 (3), Corollary 3.12]{coppepan}
that the cohomology functor induces an isomorphism 
$\Hom_{\sfK(R)}(M, J) \xrightarrow{\cong} \Hom_{\GrMod H}(\tuH(M), \tuH(J))$. 
\end{proof}

\begin{lemma}\label{201711302311}
Let $M \in \sfD^{> -\infty}(R)$ 
and $g: \tuH(M) \to J $ be a monomorphism in $\GrMod H$ 
with $J$ is injective such that $J^{i} = 0$ for $i < \inf M$. 
Then, 
\begin{enumerate}[(1)]
\item 
$J$ is a shifted direct sum of objects in $\Inj^{0} H$. 

\item 
There exists  a shifted direct sum $I'$ of objects in $\Inj^{0} R$ 
such that $\tuH(I') = J$. 

\item  
There exists a $\#$-injective contractible DG-$R$-module $I''$ 
and a monomorphsim $\underline{g}: M \to I' \oplus I''$ 
such that we have  $\tuH(\underline{g}) = g$ 
under the identification 
$\tuH(I'\oplus I'') = J$.
\end{enumerate}
\end{lemma}

\begin{proof}
(1) is \cite[Lemma 2.7]{adasore}. 
(2) is proved in \cite[Section 3.2]{coppepan}
(3) can be proved by the same way of usual DG-injective resolution. 
\end{proof}

Let $M \in \sfD^{> -\infty}(R)$. 
We take an injective resolution $J_{\bullet}$ 
of $\tuH(M)$ inside $\GrMod H$ with $J_{-i}$ a shifted direct sum of objects in $\Inj^{0} H$. 
\[
0\to\tuH(M) \to J_{0} \to J_{-1} \to \cdots \to J_{-i+1} \to J_{-i} \to \cdots 
\]
By Lemma \ref{201711302311}.2, 
there exists $I'_{-i}$ which is a shifted direct sum of objects in $\Inj^{0} R$ 
such that $\tuH(I'_{-i}) =J_{-i}$.  

Using Lemma \ref{201711302311}.3, 
we can inductively construct a contractible DG-$R$-module 
$I''_{-i}$ which is $\#$-injective 
for $i \geq 0$ and an exact sequence inside $\sfC(R)$ 
\[
 0\to M \to I_{0} \to I_{-1} \to \cdots \to I_{ -i+1} \to I_{ -i}
\]
where $I_{-i} := I'_{-i} \oplus I''_{-i}$.

The following is an injective version of Proposition \ref{CE bounded proposition}. 
Let $M $ be a graded $H$-module. 
We denote by $\grinjdim_{H} M $ the injective dimension as graded $H$-module. 
In other words, $\grinjdim_{H} M$ is the injective dimension as an object of the abelian category $\GrMod H$.

\begin{proposition}\label{201711232212}.
Assume that $H^{\ll 0} = 0$. 
Let $M\in \sfD^{\mrb}(R)$. 
Then, if $\grinjdim_{H} \tuH(M) < \infty$, 
then $\injdim_{R} M < \infty$. 
\end{proposition}

We note that the converse of the above statement is not true. 
Indeed, Example \ref{hanrei}  also gives    such  an  example.

%
%
%

\end{document}